\documentclass[11pt]{amsart}
\usepackage{geometry}                % See geometry.pdf to learn the layout options. There are lots.
\usepackage{graphicx}
\usepackage{amssymb}
\usepackage{epstopdf}
\usepackage{amsmath}
\usepackage{color}
\usepackage{hyperref}
\usepackage{pdfsync}

\usepackage[all]{xy}
\xyoption{matrix}
\xyoption{arrow}

\newtheorem{thm}{Theorem}[section]
\newtheorem{lem}[thm]{Lemma}
\newtheorem{cor}[thm]{Corollary}
\newtheorem{prop}[thm]{Proposition}

\theoremstyle{definition}
\newtheorem{defn}[thm]{Definition}
\newtheorem{eg}[thm]{Example}

\newtheorem{rem}[thm]{Remark}
 
\numberwithin{equation}{section}

\newcommand{\brk}[1]{\left<#1\right>}

\newcommand{\vs}[1]{\vskip .#1 cm} %enter amount of skip wanted at #1
\newcommand{\noi}{\noindent}

\newcommand{\xrarrow}{\xrightarrow} %right arrow {label on top}

\newcommand{\ot}{\leftarrow}
\newcommand{\then}{\Rightarrow}
\newcommand{\neht}{\Leftarrow}

\newcommand{\f}{\varphi}
 
\newcommand{\into}{\hookrightarrow}
 \newcommand{\onto}{\twoheadrightarrow}
 
 \newcommand{\cof}{\rightarrowtail}

\newcommand{\smallcoprod}{\,{\textstyle{\coprod}}\,}
\newcommand{\smallprod}{\,{\textstyle{\prod}}\,}

\DeclareMathOperator{\simp}{simp}%\newcommand{\simp}{\text{simp}}
\DeclareMathOperator{\coker}{coker}%\newcommand{\coker}{\text{coker}}
\DeclareMathOperator{\Hom}{Hom}%
\DeclareMathOperator{\Ext}{Ext}%
\DeclareMathOperator{\End}{End}%
\DeclareMathOperator{\undim}{\underline{dim}}
 
\newcommand{\field}[1]{\mathbb{#1}}
\newcommand{\ZZ}{\ensuremath{{\field{Z}}}}
\newcommand{\CC}{\ensuremath{{\field{C}}}}
\newcommand{\RR}{\ensuremath{{\field{R}}}}
\newcommand{\QQ}{\ensuremath{{\field{Q}}}}
\newcommand{\NN}{\ensuremath{{\field{N}}}}

\newcommand{\commentout}[1]{}

\newcommand{\cA}{\ensuremath{{\mathcal{A}}}}

\newcommand{\cB}{\ensuremath{{\mathcal{B}}}}
\newcommand{\cC}{\ensuremath{{\mathcal{C}}}}
\newcommand{\cD}{\ensuremath{{\mathcal{D}}}}

\newcommand{\cG}{\ensuremath{{\mathcal{G}}}}
\newcommand{\cH}{\ensuremath{{\mathcal{H}}}}
\newcommand{\cI}{\ensuremath{{\mathcal{I}}}}

\newcommand{\cK}{\ensuremath{{\mathcal{K}}}}

\newcommand{\cN}{\ensuremath{{\mathcal{N}}}}

\newcommand{\cS}{\ensuremath{{\mathcal{S}}}}

 \newcommand\vare{\varepsilon}

\title{Signed exceptional sequences and the cluster morphism category}
\author{Kiyoshi Igusa}
\address{Department of Mathematics, Brandeis University, Waltham, MA 02454}\email{igusa@brandeis.edu}

\author{Gordana Todorov}
\address{Department of Mathematics, Northeastern University, Boston, MA 02115}\email{g.todorov@northeastern.edu}


\begin{document}

\begin{abstract} We introduce signed exceptional sequences as factorizations of morphisms in the cluster morphism category. The objects of this category are wide subcategories of the module category of a hereditary algebra. A morphism $[T]:\cA\to\cB$ is the equivalence class of a rigid object $T$ in the cluster category of $\cA$ so that $\cB$ is the right hom-ext perpendicular category of the underlying object $|T|\in\cA$. Factorizations of a morphism $[T]$ are given by total orderings of the components of $T$. This is equivalent to a ``signed exceptional sequence.'' For an algebra of finite representation type, the geometric realization of the cluster morphism category is an Eilenberg-MacLane space with fundamental group equal to the ``picture group'' introduced by the authors in \cite{IOTW4}.
\end{abstract}

\maketitle

\tableofcontents

%\listoffigures
 %-----------------------------------------------------------------------------------------
% INTRODUCTION
%-----------------------------------------------------------------------------------------

\section*{Introduction}

The purpose of this paper is to give an algebraic version of some of the topological definitions, statements and proofs in our joint paper with Kent Orr and Jerzy Weyman about the picture groups for Dynkin quivers \cite{IOTW4}. To avoid repetition, the concurrently written paper \cite{IOTW4} will logically depend on this paper. In the last section of this paper we briefly review, extend and simplify the ideas from earlier versions of \cite{IOTW4} to lay the background for a more streamlined revision of that paper. The conversion to algebra follows the ideas of Quillen \cite{Quillen}. Topological spaces are replaced with small categories, continuous maps with functors and homotopies with natural transformations. In particular, a finite CW-complex can, up to homotopy, be represented algebraically as a finite category, namely, one having finitely many objects and finitely many morphisms between any two objects. When this process is applied to the CW-complex associated in \cite{IOTW4} to a Dynkin quiver, we obtain a category whose morphisms are given by signed exceptional sequences.

Let $\Lambda$ be a finite dimensional hereditary algebra over any field. Then the \emph{cluster morphism category} $\cG(\Lambda)$ of $\Lambda$ is defined to be the category whose objects are the finitely generated wide subcategories of $mod\text-\Lambda$ [InTh], (Section \ref{ss 1.1: wide subcategories} below). Such a subcategory $\cA\subseteq mod\text-\Lambda$ is hereditary and abelian and has a cluster category which we denote by $\cC_\cA$ \cite{BMRRT}. For any indecomposable object $T$ in the cluster category, let $|T|\in\cA$ be the \emph{underlying module} of $T$ given by $|M|=M$ if $T=M$ is a module and $|X[1]|=X$ for shifted objects $X[1]$ where $X$ is an object in $\cA$ which is projective in $\cA$ but not necessarily projective in $mod\text-\Lambda$. We extend additively to all objects of $\cC_\cA$ and to all objects of $\cA\cup \cA[1]\subset \cD^b(\cA)$. Then $|T|\in\cA$ is well defined up to isomorphism for any $T\in\cC_\cA$. The \emph{rank} of $\cA$, denoted $rk\,\cA$, is defined to be the number of nonisomorphic simple objects of $\cA$.

Recall that $T\in\cC_\cA$ is \emph{rigid} if $\Ext_{\cC_\cA}^1(T,T)=0$. We say that two rigid objects $T,T'$ are \emph{equivalent} if $add\,T=add\,T'$, i.e., $T,T'$ have isomorphic summands. Given $\cA,\cB\in \cG(\Lambda)$ a morphism $[T]:\cA\to \cB$ is defined to be the equivalence class of a rigid object $T\in \cC_\cA$ with the property that $|T|^\perp\cap \cA=\cB$ where $M^\perp$ is the right hom-ext-perpendicular category of $M$ in $mod\text-\Lambda$. We note that, if $\Lambda$ has finite representation type, then the cluster morphism category of $\Lambda$ has finitely many objects and finitely many morphism.

The last part of the definition of the cluster morphism category is the definition of composition of morphisms. This is a difficult technical point which requires a change in terminology from equivalence classes of rigid objects of cluster categories to partial cluster tilting sets (Definition \ref{def: cluster tilting set}). The composition of $[T]:\cA\to\cB$ and $[S]:\cB\to \cB'$ is given by $[\sigma_TS\coprod T]:\cA\to\cB'$ where $\sigma_TS\in\cC_\cA$ is the unique (up to isomorphism) rigid object in $\cC_\cA$ having the following two properties.
\begin{enumerate}
\item $\sigma_TS\coprod T$ is a rigid object in $\cC_\cA$.
\item $\undim \sigma_TS-\undim S$ is a linear combination of $\undim T_i$ where $T=\coprod_iT_i$.
\end{enumerate}
We were not able to construct a functor $\sigma_T:\cC_\cB\to\cC_\cA$ realizing this mapping defined on rigid objects of $\cC_\cB$. What we construct in this paper is a mapping
\[
	\sigma_T:\cC(\cB)\to\cC(\cA)
\]
from the set $\cC(\cB)$ of isomorphism classes of rigid indecomposable objects of $\cC_\cB$ to $\cC(\cA)$. With this in mind, we shift our notation and use \emph{partial cluster tilting sets} $T=\{T_1,\cdots,T_k\}\subset \cC(\cA)$  (Definition \ref{def: cluster tilting set}) which are sets of components of rigid objects of $\cC_\cA$. We say that $T$ is a \emph{cluster tilting set} if $k$ is maximal ($k=rk\,\cA$). With this notation, morphisms are written $[T_1,\cdots,T_k]:\cA\to\cB$ and composition of morphisms is written 
\[
	[S_1,S_2,\cdots,S_\ell]\circ[T_1,\cdots,T_k]=[\sigma_TS_1,\sigma_TS_2,\cdots,\sigma_TS_\ell,T_1,\cdots,T_k]:\cA\to\cB'.
\]

The \emph{rank} of a morphism $[T]:\cA\to\cB$ is defined to be the number of elements of $T$ as a subset of $\cC(\cA)$ (the number of nonisomorphic components of $T$ as object of $\cC_\cA$). Then $rk\,[T]=rk\,\cA-rk\,\cB$. So, $[T]$ has maximal rank if and only if $T$ is a cluster tilting set in $\cC(\cA)$. A \emph{signed exceptional sequence} can be defined to be a sequence of objects $(X_1,\cdots,X_k)$ in $mod\text-\Lambda\cup mod\text-\Lambda[1]\subset \cD^b(mod\text-\Lambda)$ with the property that
\[
	[X_1]\circ[X_2]\circ\cdots\circ [X_k]:mod\text-\Lambda\to \cB
\]
is a sequence of composable morphisms in $\cG(\Lambda)$ of rank 1 from $mod\text-\Lambda$ to $\cB=\bigcap |X_i|^\perp$. This is equivalent to the following. 
\begin{defn}\label{def: signed exceptional sequence}[Subsection \ref{ss 2.1: Def of signed exceptional sequence}]
A \emph{signed exceptional sequence} in a wide subcategory $\cA\subseteq mod\text-\Lambda$ is a sequence of objects $X_1,\cdots,X_k$ in $\cA\cup \cA[1]$ satisfying the following.% two properties.
\begin{enumerate}
\item $(|X_1|,|X_2|,\cdots,|X_k|)$ is an exceptional sequence in $\cA$
\item $X_i\in\cC(\cA_i)$ where $|X_{i+1}\coprod \cdots\coprod X_k|^\perp=\cA_i$, i.e., either $X_i\in\cA_i$ or $X_i=P[1]$ where $P$ is an indecomposable projective object of $\cA_i$.
\end{enumerate}
The signed exceptional sequence is called \emph{complete} if $k$ is maximal, i.e., $k=rk\,\cA$.
\end{defn}

Consider totally ordered {cluster tilting sets} $(T_i)=(T_1,\cdots,T_k)$ in $\cC(\cA)$. We refer to these as \emph{ordered cluster tilting sets}.

\begin{thm}[Theorem \ref{thm 2.3: bijection one}]\label{thm 2.3}
There is a bijection between the set of ordered cluster tilting sets and the set of (complete) signed exceptional sequences.
\end{thm}

For example, in type $A_2$ the cardinality of this set is $2!C_3=2\cdot 5=10$. Another example is the sequence of simple modules $(S_n,\cdots,S_2,S_1)$ in reverse admissible order (so that $S_n$ is injective and $S_1$ is projective). Since each $S_k$ is projective in the right perpendicular category of $S_{k-1},\cdots,S_1$, it can have either sign. So, there are $2^n$ possible signs. It is easy to see that the corresponding ordered cluster tilting sets are distinct as unordered cluster tilting sets. (Proposition \ref{prop: 2 to n distinct clusters}.)

Our sign conventions make the dimension vectors of the objects in certain signed exceptional sequences into the negatives of the $c$-vectors of cluster tilting objects. Speyer and Thomas \cite{ST} gave a characterization of $c$-vectors. We give another description which also determines the cluster tilting object corresponding to the $c$-vectors.

\begin{thm}[Theorem \ref{thm: Which sig exc seqs are c-vectors?}]\label{thm 2.13}
The dimension vectors of objects $X_i$ in a signed exceptional sequence form the set of negative $c$-vectors of some cluster tilting object $T$ if and only if the ordered cluster tilting set $(T_i)=(T_1,\cdots,T_n)$ corresponding to $(X_i)$ under the bijection of Theorem \ref{thm 2.3: bijection one} has the property that $\Hom_\Lambda(|T_i|,|T_j|)=0=\Ext^1_\Lambda(|T_i|,|T_j|)$ for $i<j$. Furthermore, all sets of (negative) $c$-vectors are given in this way and $T=\coprod_iT_i$.
\end{thm}

The equation $T=\coprod_iT_i$ means we have two different descriptions of the same bijection:
\[
	\{\text{signed exceptional sequences $(X_i)$ s.t. $-\undim X_i$ are $c$-vectors}\}
\]
\[
	\cong\{
	\text{ordered cluster tilting sets $(T_i)$ s.t. $\Hom_\Lambda(|T_i|,|T_j|)=0=\Ext^1_\Lambda(|T_i|,|T_j|)$ for $i<j$}
	\}
\]
One bijection is given by sending $(X_i)$ to the ordered set of $c$-vectors $(-\undim X_i)$ and then to the {ordered cluster tilting set} which corresponds to these in the usual way by, e.g., Equation \eqref{eq characterizing c-vectors} in section \ref{ss 2.4: c-vectors} below. The other bijection is given by restriction of the bijection given in Theorem \ref{thm 2.3: bijection one}.

Finally, we return to the motivation of this paper which is to show the following.

\begin{thm}[Theorem \ref{thm 3.1: 2nd main theorem}]\label{thm 3.1: Intro}
The classifying space $B\cG(\Lambda)$ of the cluster morphism category of a hereditary algebra of finite representation type is a $K(\pi,1)$ where $\pi$ is the ``picture group'' introduced in \cite{IOTW4}. In fact $B\cG(\Lambda)$ is homeomorphic to the topological space $X(\Lambda)$ constructed in \cite{IOTW4}.
\end{thm}

This gives a proof of the fact that the ``picture space'' $X(Q)$ is a $K(\pi,1)$ for any Dynkin quiver $Q$. A proof of the following slightly stronger theorem, using the results of this paper and ideas from \cite{IIT} will appear in a future paper: For $\Lambda$ of finite type, $B\cG(\Lambda)$ is a ``non-positively curved cube complex'' and therefore the picture group is a ``CAT(0)-group''. Contrarily, for $\Lambda$ of tame infinite type, $B\cG(\Lambda)$ is not a $K(\pi,1)$.

The contents of this paper are as follows. In Section \ref{ss 1.1: wide subcategories} we give the basic definitions including the key definitions (\ref{def 1.5: A(alpha)}, \ref{def 1.7: cluster morphism}) of $\cA(\alpha_\ast)$ and cluster morphisms $[T]:\cA\to \cB$ as outlined above. In Section \ref{ss 1.2: composition of cluster morphisms} we give the definition of composition of cluster morphisms assuming Proposition \ref{prop 1.8: Properties of sigma_T} which is proved in Section \ref{ss 1.3: proof of properties of sT} using \cite{IOTW1} and \cite{IOTW3}. In Section \ref{sec 2: signed exceptional sequences} we define signed exceptional sequences and show that they have the properties outlined above. 

In Section \ref{sec 3: classifying space of G(S)} we prove the second main Theorem \ref{thm 3.1: 2nd main theorem} that the classifying space of the cluster morphism category is a $K(\pi,1)$. First, we state the extension of the theorem (Theorem \ref{thm 3.5: G(S) is K(pi,1)}) to any convex set of roots (Definition \ref{def: convex set of roots}). In Section \ref{ss 3.2: outline of G(S)} we give an outline of the proof of Theorem \ref{thm 3.5: G(S) is K(pi,1)} using HNN extensions. The details occupy the rest of Section \ref{sec 3: classifying space of G(S)}.

In Section \ref{ss 4.1: the CW-complex X(S)} we recall the picture space $X(\Lambda)$ of a hereditary algebra $\Lambda$ of finite representation type and extend the definition to any finite convex set of roots $\cS$. This space is a finite CW-complex with one cell $e(\cA)$ for every wide subcategory $\cA$ in $mod\text-\Lambda$. Section \ref{ss 4.2: proof that X(S)=BG(S)} proves that $X(\cS)$ is homeomorphic to $B\cG(\cS)$. Section \ref{ss 4.3: example} gives a simple example of the correspondence between parts of $X(\cS)$ and parts of $B\cG(\cS)$. Finally, in \ref{ss 4.4: semi-invariants}, we construct a codimension one subcomplex $D(\cS)\subseteq B\cG(\cS)$ and show in Proposition \ref{prop: DA(b)=e(A) cap D(b)} that $D(\cS)$ is the category theoretic version of the picture complex $L(\cA)\subset S^{n-1}$.

\section{Definition of cluster morphism category}

We will construct a category abstractly by defining objects to be finitely generated wide categories. We call it the ``cluster morphism category'' since its morphisms are (isomorphism classes of) partial cluster tilting objects.

\subsection{Wide subcategories}\label{ss 1.1: wide subcategories}

Suppose that $\Lambda$ is a hereditary finite dimensional algebra over a field $K$ which we assume to be infinite. Let $mod$-$\Lambda$ be the category of finite dimensional right $\Lambda$-modules. Then a \emph{wide subcategory} of $mod$-$\Lambda$ is defined to be an exactly embedded abelian subcategory $\cA$ of $mod$-$\Lambda$ which is closed under extensions. In particular, taking extensions with $0$, any module which is isomorphic to an object of $\cA$ is already in $\cA$. A wide category is called \emph{finitely generated} if there is one object $P$, which we can take to be projective, so that every other object $X$ of $\cA$ is a quotient of $P^m$ for some $m$ depending on $X$. The wide category $\cA$ is then isomorphic to the category of finitely generated right modules over the endomorphism ring of $P$. This is an hereditary finite dimensional algebra over the ground field. 

\begin{thm}\cite{Ingalls-Thomas} There is a 1-1 correspondence between finitely generated wide subcategories in $mod$-$\Lambda$ and isomorphism classes of cluster tilting objects in the cluster category of $\Lambda$.
\end{thm}

In this section, we will review the well-known correspondence between cluster tilting objects of the cluster category with support tilting modules.

We recall that the \emph{quiver} of $\Lambda$ consists of one vertex for every (isomorphism class of) simple module $S_i$ for $i=1,\cdots,n$ and one arrow $i\to j$ if $\Ext^1_\Lambda(S_i,S_j)\neq0$. We number these in \emph{admissible order} which means that $\Ext^1_\Lambda(S_i,S_j)= 0$ if $i<j$. Let $P_i,I_i$ be the projective cover and injective envelope of $S_i$ respectively. Let $F_i=\End_\Lambda(S_i)=\End_\Lambda(P_i)=\End_\Lambda(I_i)$. This is a division algebra which acts on the left on all three of these modules. So, we identify $F_i$ with these endomorphism rings making them all equal. The modules $S_i,P_i,I_i$ are {exceptional} where $X$ is called \emph{exceptional} if $\End_\Lambda(X)$ is a division algebra and $\Ext^1_\Lambda(X,X)=0$.

The \emph{support} of $M$ is the set of vertices $i$ for which $\Hom_\Lambda(P_i,M)\neq0$. A \emph{(basic) support tilting module} is a module $M$ so that\begin{enumerate}
\item $M$ is a direct sum of $k$ nonisomorphic exceptional modules $M_i$ where $k$ is the size of the support of $M$.
\item $\Ext^1_\Lambda(M,M)=0$.
\end{enumerate}

For each support tilting module $M=M_1\coprod\cdots\coprod M_k$ there is a unique cluster tilting set (up to isomorphism) which is the unordered set of objects $\{M_1,M_2,\cdots,M_k\}$ union the $n-k$ shifted projective modules $P_j[1]$ for all $j$ not in the support of $M$. %For the moment 
We will take this to be the definition of a {cluster tilting set}.

\begin{defn}\label{def: cluster tilting set}
Suppose that $\cA$ is a finitely generated wide subcategory of $mod$-$\Lambda$ with $k$ nonisomorphic projective objects $Q_1,\cdots,Q_k$. Since these may not be projective in $mod$-$\Lambda$ we sometimes refer to them as \emph{relative projective objects}. By a \emph{partial cluster tilting set} for $\cA$ we mean a set of objects $T_1,\cdots,T_\ell$ in the bounded derived category of $\cA$ so that
\begin{enumerate}
\item Each $T_i$ is either a shifted projective object $Q_j[1]$ or an exceptional object of $\cA$.
\item For all $i,j$ we have: $\Ext_{\cD^b}^1(T_i,T_j)=0$. Equivalently:
\begin{enumerate}
\item $\Ext_\Lambda^1(T_i,T_j)$ if $T_i,T_j$ are modules.
\item $\Hom_\Lambda(Q,T_j)=0$ if $T_i=Q[1]$ and $T_j$ is a module.
\end{enumerate}
\end{enumerate}
If $\ell=k$ the partial cluster tilting set is called a \emph{cluster tilting set}. We view all shifted projective objects $Q[1]$ as objects of the bounded derived category of $mod$-$\Lambda$. We use the notation $|T|$ to denote the \emph{underlying module} of $T$ which is equal to $T$ if $T$ is a module and $|Q[1]|=Q$.
\end{defn}

We denote a finitely generated wide subcategory by its set of simple objects. Thus $\cA(M_1,\cdots,M_k)$ denotes the wide subcategory of $mod$-$\Lambda$ whose simple objects are $M_1,\cdots,M_k$.

\begin{prop}
A finite set of exceptional modules $\{M_1,\cdots,M_k\}$ forms the set of simple objects in a finitely generated wide subcategory of $mod$-$\Lambda$ if and only if it satisfies the following two conditions.
\begin{enumerate}
\item $\Hom_\Lambda(M_i,M_j)=0$ for all $i\neq j$.
\item The modules $M_i$ can be ordered in such a way that $\Ext^1_\Lambda(M_i,M_j)=0$ for all $1\le i<j\le k$.
\end{enumerate}
\end{prop}

We say that the $M_i$ are \emph{hom-orthogonal} if they satisfy (1). Note that, given (1), (2) is equivalent to the statement that $(M_k,\cdots,M_1)$ is an exceptional sequence.

\begin{proof}
Necessity is clear. Conversely, suppose these condition hold. Then the exceptional sequence $(M_k,\cdots,M_1)$ can be completed by adding $\Lambda$-modules $M_n,\cdots,M_{k+1}$ on the left. Then $M_1,\cdots,M_k$ are the simple objects of the wide subcategory $(M_n\coprod \cdots\coprod M_{k+1})^\perp$.
\end{proof}

The \emph{dimension vector} $\underline\dim M\in\NN^n$ of a module $M$ is defined to be the integer vector whose $i$th coordinate is $\dim_{F_i}\Hom_\Lambda(P_i,M)$. The dimension vector of any shifted object $M[1]$ is defined to be $\underline\dim (M[1])=-\underline\dim M$. The \emph{Euler-Ringel form} $\brk{\cdot,\cdot}$ is the bilinear form on $\ZZ^n$ with the property that
\[
	\brk{\underline\dim M,\underline\dim N}=\dim_K\Hom_\Lambda(M,N)-\dim_K\Ext_\Lambda^1(M,N)
\]
If $M,N$ lie in a finitely generated wide subcategory $\cA$ then this form takes the same value if evaluated in $\cA$ or in $mod$-$\Lambda$ because $\cA\into mod\text-\Lambda$ is an exact full embedding (so, $\Hom_\cA(M,N)=\Hom_\Lambda(M,N)$ for all $M,N\in\cA$) and $\cA$ is extension closed in $mod$-$\Lambda$ (so, $\Ext^1_\cA(M,N)=\Ext^1_\Lambda(M,N)$ for all $M,N\in\cA$).

We will also use the same bilinear form in the derived category using the following formula which is easily verified.

\begin{prop}
Suppose that $M,N$ lie in $\cD^b(\cA)$. Then
\[
	\brk{\underline\dim M,\underline\dim N}=\sum_{j\in\ZZ}(-1)^j\dim_K\Ext^j_{\cD^b(\cA)}(M,N)=\sum_{j\in\ZZ}(-1)^j\dim_K\Ext^j_{\cD^b(\Lambda)}(M,N)
\]
\end{prop}

Recall that the dimension vectors of all exceptional objects and all shifted relative projective objects of f.g. wide subcategories are {real Schur roots} and all real Schur roots occur as such \cite{Ringel}. For example, let $\beta$ be a real Schur root of $\Lambda$. Let $M_\beta$ be the unique exceptional object with dimension vector $\beta$. Then $M_\beta$ is a relative projective object in the abelian category $\cA(M_\beta)$ generated by $M_\beta$. So, both $\beta$ and $-\beta$ occur as dimension vectors of exceptional objects and shifted relative projective objects in some f.g. wide subcategory of $mod$-$\Lambda$.

\begin{defn}\label{def 1.5: A(alpha)}
Let $\alpha_\ast=\{\alpha_1,\alpha_2,\cdots,\alpha_k\}$ be an unordered set of distinct positive real Schur roots so that the corresponding modules $M_1,\cdots,M_k$ are hom-orthogonal and form an exceptional sequence in some order. Then we denote by $\cA(\alpha_\ast)$ the wide subcategory with simple objects $M_i$. Equivalently, $\cA(\alpha_\ast)$ is the abelian category of all modules having a filtration for which all subquotients are isomorphic to some $M_i$. Let $\cC(\alpha_\ast)$ be the union of the set of all exceptional objects of $\cA(\alpha_\ast)$ and the set of shifted relative projective objects $Q[1]$ for all indecomposable relative projective objects $Q$ in $\cA(\alpha_\ast)$. Two elements $T,T'$ of $\cC(\alpha_\ast)$ are called \emph{ext-orthogonal} if $\Ext^1_{\cD^b}(T,T')=\Ext^1_{\cD^b}(T',T)=0$.
\end{defn}

\begin{defn}\label{def 1.6: perpendicular category}
For any finitely generated wide subcategory $\cA$ in $mod$-$\Lambda$ let $\,^\perp \cA$ denote the full subcategory of $mod$-$\Lambda$ of all modules $X$ with the property that $\Hom_\Lambda(X,M)=0=\Ext^1_\Lambda(X,M)$ for all $M\in\cA$. Similarly, let $\cA^\perp$ be the full subcategory of $mod$-$\Lambda$ of all modules $X$ with the property that $\Hom_\Lambda(M,X)=0=\Ext^1_\Lambda(M,X)$ for all $M\in\cA$.
\end{defn}

It is well-known that the categories $\,^\perp \cA$ and $\cA^\perp$ are finitely generated wide subcategories of $mod$-$\Lambda$. As a special case (replacing $mod$-$\Lambda$ with $\cB$), $\cB\cap(^\perp\cA)$ and $\cB\cap(\cA^\perp)$ are finitely generated wide subcategories of $\cB$ if $\cA\subseteq \cB$.

\begin{defn}\label{def 1.7: cluster morphism}
Suppose that $\cA$ and $\cB$ are finitely generated wide subcategories of $mod$-$\Lambda$ and $\cB\subseteq \cA$. Then a \emph{cluster morphism} $\cA\to \cB$ is defined to be a partial cluster tilting set $T=\{T_1,\cdots,T_k\}$ in $\cC(\cA)$ so that $|T|^\perp\cap \cA=\cB$. In other words, $\cB$ is the full subcategory of $\cA$ of all objects $B$ so that $\Hom_\Lambda(|T_i|,B)=0=\Ext^1_\Lambda(|T_i|,B)$ for all $i$. We denote the corresponding morphism by $[T]$ or $[T_1,\cdots,T_k]:\cA\to\cB$. Note that $T$ is an unordered set. For example, the empty set gives the identity morphism $[\,]=id_\cA:\cA\to\cA$.
\end{defn}

\subsection{Composition of cluster morphisms}\label{ss 1.2: composition of cluster morphisms}
We come to the difficult part of the definition which is the formula for composition of cluster morphisms. Suppose that we have cluster morphisms $[T]:\cA(\alpha_\ast)\to \cA(\beta_\ast)$ and $[S]:\cA(\beta_\ast)\to \cA(\gamma_\ast)$. Then the composition $[S]\circ[T]:\cA(\alpha_\ast)\to  \cA(\gamma_\ast)$ will be the partial cluster tilting set
\begin{equation}\label{eq:composition of cluster morphisms}
	[S_1,\cdots,S_\ell]\circ[T_1,\cdots,T_k]=[\sigma_TS_1,\cdots,\sigma_TS_\ell, T_1,\cdots,T_k]
\end{equation}
where the set mapping $\sigma_T:\cC(\beta_\ast)\to \cC(\alpha_\ast)$ is uniquely determined by the following proposition.

\begin{prop}\label{prop 1.8: Properties of sigma_T}
Suppose that $[T]=[T_1,\cdots,T_k]$ is a cluster morphism $\cA(\alpha_\ast)\to \cA(\beta_\ast)$. Then, for any $S\in \cC(\beta_\ast)$ there is a unique object $\sigma_TS\in\cC(\alpha_\ast)$ satisfying the following three conditions.
\begin{enumerate}
\item[(a)] $\{T_1,\cdots,T_k,\sigma_TS\}$ is a partial cluster tilting set in $\cC(\alpha_\ast)$.
\item[(b)] $\cA(\beta_\ast)\cap |S|^\perp=\cA(\beta_\ast)\cap |\sigma_TS|^\perp$
\item[(c)] $\underline\dim(\sigma_TS)-\underline\dim S$ is an integer linear combination of the vectors $\underline\dim T_i$.
\end{enumerate}
Furthermore, the following additional properties hold as a consequence of the first three.
\begin{enumerate}
\item[(d)] If $S_1,S_2$ are ext-orthogonal elements of $\cC(\beta_\ast)$ then $\sigma_TS_1,\sigma_TS_2$ are ext-orthogonal elements of $\cC(\alpha_\ast)$.
\item[(e)] If $\{T_1,\cdots,T_k,S\}$ is a partial cluster tilting set in $\cC(\alpha_\ast)$ then $\sigma_TS=S$.
\end{enumerate}
\end{prop}

We note that Property (e) follows immediately from the uniqueness of $\sigma_TS$. The proof of the other statements will be given later. For the moment suppose that this proposition holds. Then we will show that composition of cluster morphisms is associative. But first we need to show that composition is defined.

\begin{cor}
Given cluster morphisms $[T]:\cA(\alpha_\ast)\to \cA(\beta_\ast)$ and $[S]:\cA(\beta_\ast)\to \cA(\gamma_\ast)$, the formula \eqref{eq:composition of cluster morphisms} gives a cluster morphism $[T,\sigma_TS]:\cA(\alpha_\ast)\to\cA(\gamma_\ast)$. In other words Properties (a) and (b) in the proposition above hold when $\sigma_TS=\{\sigma_TS_1,\cdots,\sigma_TS_\ell\}$ has more than one element.
\end{cor}

\begin{proof}
First, $\{T,\sigma_TS\}=\{T_1,\cdots,T_k,\sigma_TS_1,\cdots,\sigma_TS_\ell\}$ is a partial cluster tilting set in $\cC(\alpha_\ast)$ since, by (a), each $\sigma_TS_i$ is ext-orthogonal to each $T_j$ and by (d) the $\sigma_TS_i$ are ext-orthogonal to each other.  Second, $[T,\sigma_TS]$ is a morphism $\cA(\alpha_\ast)\to\cA(\gamma_\ast)$. In other words, $\cA(\alpha_\ast)\cap |T,\sigma_TS|^\perp=\cA(\gamma_\ast)$. But this follows from Property (b):
 \[
 	\cA(\gamma_\ast)=\cA(\beta_\ast)\cap |S|^\perp=\cA(\beta_\ast)\cap |S_1|^\perp\cap \cdots\cap |S_{\ell-1}|^\perp\cap |S_\ell|^\perp=\bigcap\left(
	\cA(\beta_\ast)\cap|S_i|^\perp
	\right)
 \]
 \[
 	=\bigcap\left(
	\cA(\beta_\ast)\cap|\sigma_TS_i|^\perp
	\right)=\cA(\beta_\ast)\cap |\sigma_TS|^\perp=\cA(\alpha_\ast)\cap |T,\sigma_TS|^\perp
 \]
\end{proof}

\begin{cor}
 The composition law \eqref{eq:composition of cluster morphisms} is associative and unital. Consequently, we have a category with objects given by finitely generated wide subcategories $\cA$ of $mod$-$\Lambda$ and morphisms given by partial cluster tilting sets $[T]:\cA\to \cA\cap|T|^\perp$.
 \end{cor}
 
 \begin{proof} 
 
 It follows from the Definition \eqref{eq:composition of cluster morphisms} that the empty set in $\cC(\beta_\ast)$ is a left identity: $[\,]\circ [T]=[T,\sigma_T(\emptyset)]=[T]$. As a special case of  Property (e), $\sigma_\emptyset S=S$. Therefore, the empty set is a right identity:
 \[
 	[S_1,\cdots,S_\ell]\circ[\,]=[\sigma_\emptyset S_1,\cdots,\sigma_\emptyset S_\ell]=[S_1,\cdots,S_\ell]
 \]

Finally, we need to show that composition is associative.  So, suppose we have the composable cluster morphisms:
\[
	\cA(\alpha_\ast)\xrarrow{[T]}
	\cA(\beta_\ast)\xrarrow{[S]}
	\cA(\gamma_\ast)\xrarrow{[R]}
	\cA(\delta_\ast)
\]
By definition we have:
\[
	([R]\circ[S])\circ[T]=[S,\sigma_SR]\circ T=[T,\sigma_TS,\sigma_T\sigma_SR]
\]
\[
	[R]\circ([S]\circ[T])=[R]\circ[T,\sigma_TS]=[T,\sigma_TS,\sigma_{T,\sigma_TS}R]
\]
Therefore, we need to show that, for each $R_i$ in $R$, $\sigma_T\sigma_SR_i=\sigma_{T,\sigma_TS}R_i$. To prove this we can assume that $R$ has only one element. Then we will verify that $\sigma_T\sigma_SR$ satisfies the three conditions which uniquely characterize $\sigma_{T,\sigma_TS}R$. By the previous corollary we have the first two conditions:
\begin{enumerate}
\item[(a)] $\{T,\sigma_TS,\sigma_T\sigma_SR\}$ forms a partial cluster tilting set in $\cC(\alpha_\ast)$ and
\item[(b)] $\cA(\gamma_\ast)\cap |R|^\perp=\cA(\beta_\ast)\cap |S,\sigma_SR|^\perp=\cA(\alpha_\ast)\cap |T,\sigma_TS,\sigma_T\sigma_SR|^\perp=\cA(\gamma_\ast)\cap |\sigma_T\sigma_SR|^\perp$
\end{enumerate}
The third condition is also easy:
\[
	\sigma_T\sigma_SR-R=(\sigma_T\sigma_SR-\sigma_SR)+(\sigma_SR-R)
\]
which is an additive combination of $\underline\dim T_i$ plus an additive combination of $\underline\dim S_j$. However, modulo the vectors $\underline\dim T_i$, each $\underline\dim S_j$ is congruent to $\underline\dim \sigma_TS_j$. Therefore:
\begin{enumerate}
\item[(c)] $\sigma_T\sigma_SR-R$ is an integer linear combination of the vectors $\underline\dim T_i$ and $\underline\dim \sigma_TS_j$.
\end{enumerate}
Therefore, by the uniqueness clause in the Proposition, we have 
\begin{equation}\label{eq: sigma TS=sigma T sigma S}
\sigma_T\sigma_SR=\sigma_{T,\sigma_TS}R
\end{equation}
making composition of cluster morphisms associative.
 \end{proof}
 
\subsection{Proof of Proposition \ref{prop 1.8: Properties of sigma_T}}\label{ss 1.3: proof of properties of sT}

To complete the definition of the cluster morphism category we need to prove Proposition \ref{prop 1.8: Properties of sigma_T}. We do this by induction on $k$ starting with $k=1$. Without loss of generality we assume that $\cA(\alpha_\ast)=mod$-$H$. Then $\cA(\beta_\ast)=|T|^\perp$.

\subsubsection{Uniqueness of $\sigma_TS$ when $k=1$}

 \begin{lem}\label{lem: X to Tm to Y means X+Y=mT}
 Let $\cA(\alpha_1,\alpha_2)$ be a finitely generated wide subcategory of $mod$-$\Lambda$ of rank $2$ and suppose that $T,X,Y \in\cC(\alpha_1,\alpha_2)$ so that $T$ is ext-orthogonal to both $X$ and $Y$. Then $\underline\dim X+\underline\dim Y$ is a multiple of $\underline\dim T$.
 \end{lem}
 
 \begin{proof}
 Cluster mutation in cluster categories of rank 2 are very well understood. After possibly switching $X$ and $Y$ we have $X=\tau Y$ and an almost split triangle
 \[
 	X\to T^m\to Y\to X[1]
 \]
 If $Y$ is not projective then $\underline\dim\, X+\underline\dim\, Y=\underline \dim\, T^m=m\,\underline \dim\, T$. If $Y$ is projective then $X=\tau Y=Y[1]$ and $\underline\dim X+\underline\dim Y=0$. So, the lemma holds in all cases.
 \end{proof}
 
We recall the statement of Proposition \ref{prop 1.8: Properties of sigma_T} when $k=1$: For any rank 1 cluster morphism $[T]:\cA(\alpha_\ast)\to\cA(\beta_\ast)$ and any $S\in\cC(\beta_\ast)$ there is a unique $\sigma_TS\in\cC(\alpha_\ast)$ so that:
 \begin{enumerate}
 \item[(a)] $\{T,\sigma_TS\}$ is a partial cluster tilting set in $\cC(\alpha_\ast)$.
\item[(b)] $\cA(\beta_\ast)\cap |S|^\perp=\cA(\beta_\ast)\cap |\sigma_TS|^\perp$
\item[(c)] $\underline\dim(\sigma_TS)-\underline\dim S$ is an integer multiple of the vector $\underline\dim T$.
 \end{enumerate}
 
 To prove uniqueness of $\sigma_TS$, let $X,Y$ be two candidates for $\sigma_TS$. Then, by Properties (a) and (b), $\{T,X\},\{T,Y\}$ are both cluster tilting sets in the rank 2 cluster category of the finitely generated wide subcategory $\,^\perp\left(|T,S|^\perp\right)$ of $\cA(\alpha_\ast)$. By Property (c), $\underline\dim\,X$ and $\underline\dim\,Y$ are both congruent to $\underline\dim\,S$ modulo $\underline\dim\,T$. By the lemma we conclude that $2\,\underline\dim\,S$ is a multiple of $\underline\dim\,T$ and thus $\underline\dim\,X$, $\underline\dim\,T$ are collinear. But this is not possible since the dimension vectors of elements of a cluster tilting set are always linearly independent.

This completes the proof of the uniqueness of $\sigma_TS$. We will now show the existence of $\sigma_TS$ satisfying Properties (a),(b),(c).

\subsubsection{Case 1: $T,S$ are modules}

We are given that $S\in T^\perp$. I.e, $(S,T)$ is an exceptional sequence.
If $\Ext^1_\Lambda(S,T)=0$ then we let $\sigma_TS=S$. This clearly satisfies all three conditions.
Otherwise, let $m\ge1$ be the dimension of $\Ext^1_\Lambda(S,T)$ over the division algebra $F_T:=\End_\Lambda(T)$. If we choose a basis for $\Ext^1_\Lambda(S,T)$ then we get an extension
\[
	T^m\cof E\onto S
\]
 which is \emph{universal} in the sense that any extension of $T$ by $S$ is given as the pushout of this extension by a unique morphism $T^m\to T$. So, in the exact sequence:
 \[
 	\Hom_\Lambda(T^m,T)\xrarrow{\cong}\Ext^1_\Lambda(S,T)\to \Ext^1_\Lambda(E,T)\to \Ext^1_\Lambda(T^m,T)=0
 \]
 the first arrow is an isomorphism making $\Ext^1_\Lambda(E,T)=0$.  Applying $\Ext^1_\Lambda(T,-)$ to the universal extension we also get $\Ext^1_\Lambda(T,E)=0$. So, $E,T$ are ext-orthogonal and we let $\sigma_TS=E$.
 
 The construction of $E$ is the well-know mutation rule for exceptional sequences. We start with the exceptional sequence $(S,T)$ and we get the exceptional sequence $(T,E)$ by the universal extension in the case when $\Ext^1_\Lambda(S,T)\neq0$. See \cite{CB} for details. In particular, $E$ is an exceptional module and $(T,E)^\perp=(T,S)^\perp$. Also, $\underline\dim\,E=\underline\dim\,S+m\,\underline\dim\,T$. So, Properties (a),(b),(c) all hold.
 
\subsubsection{Case 2: $T$ is a module and $S=Q[1]$} We are given that $(Q,T)$ is an exceptional sequence and $Q$ is a relative projective object in $T^\perp$.
 
 Suppose first that $\Hom_\Lambda(Q,T)=0$. If $Q$ is projective in $\cA(\alpha_\ast)$ then we can let $\sigma_TQ[1]=Q[1]$ and there is nothing to prove. So, suppose $Q$ is not projective in $\cA(\alpha_\ast)$. Let $T^m\cof E\onto Q$ be the universal extension. Then, just as in Case 1, $E,T$ are ext-orthogonal, $(T,E)$ is an exceptional sequence and $(T,E)^\perp=(T,S)^\perp$. We also claim that $E$ is projective in $\cA(\alpha_\ast)$. So, we can let $\sigma_TQ[1]=E[1]$ and Properties (a),(b),(c) will hold.
 
 To prove the $E$ is projective in $\cA(\alpha_\ast)$, suppose not and let $X$ be an object in $\cA(\alpha_\ast)$ of minimal length so that $\Ext^1_\Lambda(E,X)\neq0$. By right exactness of $\Ext^1_\Lambda(-,X)$, $\Ext^1_\Lambda(E,X)=0$ if $X\in \cA(\beta_\ast)$. So, we can assume $X\notin T^\perp$. This means either
 \begin{enumerate}
 \item[(i)] $\Hom_\Lambda(T,X)\neq0$ or 
 \item[(ii)] $\Ext^1_\Lambda(T,X)\neq0$. 
 \end{enumerate}
 In Case (i), let $f:T\to X$ be any nonzero morphism and let $Y$ be the cokernel of $f$. Then $\Ext^1_\Lambda(E,Y)=0$ by minimality of $X$. So, $Y=0$. But $(T,E)$ is an exceptional sequence. So, $\Ext^1_\Lambda(E,T)=0$. By right exactness of $\Ext^1_\Lambda(E,-)$ this implies that $\Ext^1_\Lambda(E,X)=0$ which is a contradiction. In Case (ii), let $X\to Z\to T^m$ be the universal extension. Applying $\Hom_\Lambda(T,-)$ to this extension we see that $Z\in T^\perp$. Therefore $\Ext^1_\Lambda(E,Z)=0$. But $\Ext^1_\Lambda(E,Z)\cong \Ext^1_\Lambda(E,X)\neq 0$ which is a contradiction. So, $E$ must be projective in $\cA(\alpha_\ast)$ as claimed.

 Now, suppose that $\Hom_\Lambda(Q,T)\neq0$. Let $f:Q\to T^m$ be the minimal left $T$-approximation of $Q$. Then, by the theory of exceptional sequences, $f$ is either a monomorphism or an epimorphism. In the first case we get a short exact sequence $Q\cof T^m\onto E$ and $(T,E)$ is an exceptional sequence with $(T,E)^\perp=(Q,T)^\perp$. By right exactness of $\Ext^1_\Lambda(T,-)$ we also get $\Ext^1_\Lambda(T,E)=0$. So, $E,T$ are ext-orthogonal and we can let $\sigma_TQ[1]=E$.

 If $f:Q\to T^m$ is an epimorphism, we let $P=\ker f$. Then $P$ is a projective object of $\cA(\alpha_\ast)$ by the same argument used to prove that $E$ is projective in the case $\Hom_\Lambda(Q,T)=0$. (Take $X$ minimal so that $\Ext^1_\Lambda(P,X)\neq 0$, then $X\notin T^\perp$ giving two cases (i), (ii) each leading to a contradiction as before.) Then we can take $\sigma_TQ[1]=P[1]$. By construction $(T,P)$ is an exceptional sequence which is braid mutation equivalent to $(Q,T)$. Therefore (a) and (b) are satisfied and (c) follows from the exact sequence $P\cof Q\onto T^m$.
 
 In all subcases of Case 2 we have the following.
 
\begin{prop}
When $T$ is a module and $S=Q[1]$, then $(T,\sigma_TS)$ is a signed exceptional sequence.
\end{prop}
 
\subsubsection{Case 3: $T=P[1]$}
 
If $S$ is a module then we have $\Hom_\Lambda(P,S)=0$. So, $P[1],S$ are ext-orthogonal and we let $\sigma_{P[1]}S=S$. This trivially satisfies Properties (a),(b),(c).
 
 So, suppose that $S=Q[1]$ where $Q$ is a relative projective object of $P^\perp$. If $Q$ is a projective object of $\cA(\alpha_\ast)$ then $P[1],Q[1]\in \cC(\alpha_\ast)$ form a partial cluster tilting set so we can let $\sigma_{P[1]}Q[1]=Q[1]$ which satisfies (a),(b),(c).
 
 We are reduced to the case when $S=Q[1]$ where $Q$ is not projective in $\cA(\alpha_\ast)$. In this case let $\dim_{F_Q}\Ext^1_\Lambda(Q,P)=m$ (necessarily positive as we will see) and let
 \[
 	P^m\cof E\onto Q
 \] 
 be the universal extension. Then we claim that $E$ is an indecomposable projective object of $\cA(\alpha_\ast)$, the proof being the same as in Case 2 above (but shorter since (ii) does not occur). So, we can let $\sigma_{P[1]}Q[1]=E[1]$ and Property (a) will hold. Since $(P,E)$ is the braid mutation of $(Q,P)$, it is an exceptional sequence. So, $E$ is indecomposable and Property (b) holds. Since $\underline\dim\,E[1]=\underline\dim\,Q[1]+m\,\underline\dim\,P[1]$, Property (c) also holds.

\subsubsection{Stronger version of Proposition \ref{prop 1.8: Properties of sigma_T}}

To complete the proof of the proposition, we need to make the statement stronger. We will prove the following theorem along with the proposition by simultaneous induction on $k$.

\begin{thm}\label{thm:sigma-T is a bijection}
 Suppose that $T=\{T_1,\cdots,T_k\}$ is a partial cluster tilting set in a finitely generated wide category $\cA(\alpha_\ast)$ of rank $k+\ell$ and let $\cA(\beta_\ast)=|T|^\perp\cap\cA(\alpha_\ast)$. Then the mapping $\sigma_T$ given by Proposition \ref{prop 1.8: Properties of sigma_T} gives a bijection
 \[
 	\sigma_T:\cC(\beta_\ast)\to \cC_T(\alpha_\ast)
 \]
 where $\cC_T(\alpha_\ast)$ is the set of all elements of $\cC(\alpha_\ast)$ which are ext-orthogonal to $T$ but not equal to any $T_i$. Furthermore, $X=\{X_1,\cdots,X_\ell\}$ is a partial cluster tilting set in $\cC(\beta_\ast)$ if and only if $\sigma_TX\cup T$ is a partial cluster tilting set in $\cC(\alpha_\ast)$.
 \end{thm}
 
 So far we have shown the existence of a unique $\sigma_T$ satisfying Properties (a),(b),(c) of Proposition \ref{prop 1.8: Properties of sigma_T} for $k=1$. We will show that this implies the theorem for $k=1$. This clearly implies Property (d) in the proposition for $k=1$. The induction step is easy for both proposition and theorem.
 
 We first note that $\sigma_T$ is clearly a monomorphism. To see this, let $\RR\alpha_\ast$ be the $\ell$ dimensional  vector space of formal real linear combinations of the roots $\alpha_i$. Then $\beta_i$ are linearly independent as elements of $\RR\alpha_\ast$ since they are dimension vectors of modules $S_i$ in an exceptional sequence. This implies that $\RR\beta_\ast\subseteq \RR\alpha_\ast$ is $\ell$ dimensional. Furthermore, the $S_i$ and $T_j$ form an exceptional sequence. So, $\beta_i$ and $\underline\dim\,T_j$ span $\RR\alpha_\ast$. So, the inclusion map $\RR\beta_\ast\into \RR\alpha_\ast$ induces a linear isomorphism
 \[
 	\lambda_T:\RR\beta_\ast\cong \RR\alpha_\ast/\RR T
 \]
By Property (c), 
\[
\underline\dim\, \sigma_TX+\RR T=\lambda_T(\underline\dim\,X)\]
for all $X\in\cC(\beta_\ast)$. Since $X$ is determined by its dimension vector, $\sigma_T$ is 1-1.
 
\subsubsection{Proof that $\sigma_T$ is a bijection for $k=1$} It remains to show that $\sigma_T$ is surjective.
 
 Let $X\in \cC_T(\alpha_\ast)$. We will find an object of $\cC(\beta_\ast)$ which maps to $X$. Let $\cA(\gamma_\ast)=|T|^\perp\cap|X|^\perp$. Then $\{T,X\}$ is a cluster tilting set in $\,^\perp\cA(\gamma_\ast)$. There exists a unique module $M$ in $\cA(\delta_\ast)=\,^\perp\cA(\gamma_\ast)$ so that $(M,|T|)$ is an exceptional sequence. Applying $\sigma_T$ we have either $\sigma_TM=X$, in which case we are done, or $\sigma_TM=Y\neq X$ and $\{T,Y\}$ is another cluster tilting set in $\cA(\delta_\ast)$ with $T$. In the second case we claim that $M$ is a relatively projective object of $\cA(\beta_\ast)$. So, $M[1]\in\cC(\beta_\ast)$ and $\sigma_TM[1]=X$ by Lemma \ref{lem: X to Tm to Y means X+Y=mT}.
 
 To prove that $M$ is projective in $\cA(\beta_\ast)$ we examine the AR quiver of the cluster category of $\cA(\delta_\ast)=\,^\perp\cA(\gamma_\ast)$. If $\cA(\delta_\ast)$ is semi-simple then $Y=M$ and $X=M[1]$. So, $M$ is projective in $\cA(\alpha_\ast)$ and thus also in $\cA(\beta_\ast)$. So, we may assume the AR quiver of the cluster category of $\cA(\delta_\ast)$ is connected:
 \[
 \xymatrixrowsep{10pt}\xymatrixcolsep{10pt}
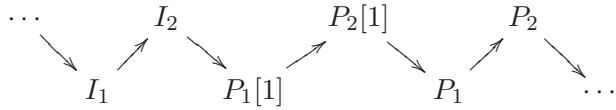
\xymatrix{%begin xy matrix
\cdots\ar[rd] && I_2\ar[rd]&& P_2[1]\ar[rd]&& P_2\ar[rd]\\
&I_1\ar[ru] &&P_1[1]\ar[ru] &&P_1\ar[ru] && \cdots
	}%end xy matrix
 \]
 We look at all possible cases.
 
 Case 0: If $T$ is not one of the four middle terms: $I_2,P_1[1],P_2[1],P_1$ then $Y=M,T,X$ form an almost split sequence $M\cof T^m\onto X$. Since $\Ext$ is right exact, $\Ext^1_\Lambda(M,-)=0$ on $T^\perp=\cA(\beta_\ast)$, making $M$ projective in that category.
 
 Case 1: If $T=P_1$ then $M=I_2$ and $\sigma_TM=P_2$ making $X=P_2[1]$. Since this is an element of $\cC(\alpha_\ast)$, $P_2$ is projective in $\cA(\alpha_\ast)$ making $P_1\subseteq P_2$ projective as well. The exact sequence $T=P_1\cof P_2^m\onto I_2=M$ show that $\Ext^1_\Lambda(I_2,-)\cong\Ext^1_\Lambda(P_2^m,-)=0$ on $T^\perp$. So, $M=I_2$ is projective in $\cA(\beta_\ast)$.
 
 Case 2: $T=P_2[1]$. Then $M=P_1=Y$ and $X=P_1[1]$. Since $P_1[1]\in\cC(\alpha_\ast)$, $M=P_1$ is projective.
 
 Case 3: If $T=P_1[1]$ then $M=I_2,X=P_2[1]$ is just like Case 1.
 
 Case 4: If $T=I_2$ then $M=Y=I_1$ and $X=P_1[1]$. So, $P_1$ is projective in $\cA(\alpha_\ast)$. The exact sequence $P_1\cof M\onto T^m$ when shows that $M$ is projective in $T^\perp=\cA(\beta_\ast)$.
 
 So, $M$ is projective in $\cA(\beta_\ast)$ in all cases and $X=\sigma_TM[1]$. So, $\sigma_T$ is a bijection for $k=1$.
 
\subsubsection{Virtual semi-invariants} To show that the bijection $\sigma_T^{-1}:\cC_T(\alpha_\ast)\to \cC(\beta_\ast)$ takes cluster tilting sets to cluster tilting sets we need some results about virtual semi-invariants.

Let $Y$ be a fixed exceptional module in $\cA(\beta_\ast)$ with dimension vector $\gamma\in\NN\beta_\ast$. We consider all pairs of relatively projective objects $P,Q$ in $\cA(\beta_\ast)$ for which there is a homomorphism $f:P\to Q$ so that
 \[
 	\Hom_\Lambda(f,Y):\Hom_\Lambda(Q,Y)\to \Hom_\Lambda(P,Y)
 \]
 is an isomorphism. When $f:P\to Q$ is a monomorphism, this is equivalent to the condition that $\Hom_\Lambda(M,Y)=0=\Ext^1_\Lambda(M,Y)$ where $M=\coker f$. 
 
 \begin{defn}\label{def: det semi-inv and supports}\cite{IOTW3}
The determinant of the matrix of $\Hom_\Lambda(f,Y)$ with respect to some basis is called a \emph{(determinantal) virtual semi-invariant} on the presentation space $\Hom_\Lambda(P,Q)$ with \emph{determinantal (det)-weight} $\gamma=\undim Y$ and is denoted $c_Y:\Hom_\Lambda(P,Q)\to K$. The set of all integer vectors $\underline\dim\,Q-\underline\dim\,P\in \ZZ\beta_\ast$ for such pairs (relatively projective objects $P,Q$ in $\cA(\beta_\ast)$ so that $c_Y$ is nonzero) is called the \emph{integer support} of $c_Y$ in $\cA(\beta_\ast)$ and is denoted $D_{\ZZ\beta_\ast}(\gamma)$. The \emph{real support} of $c_Y$, denoted $D_{\beta_\ast}(\gamma)$ is the convex hull of $D_{\ZZ\beta_\ast}(\gamma)$ in $\RR\beta_\ast$. When $\cA(\beta_\ast)=mod\text-\Lambda$ and $\RR\beta_\ast=\RR^n$, $D_{\ZZ\beta_\ast}(\gamma)$, $D_{\beta_\ast}(\gamma)$ are denoted $D_\ZZ(\gamma)$, $D(\gamma)$.
 \end{defn}

We observe that, if $X\in \cC(\beta_\ast)$ and $Y\in \cA(\beta_\ast)$ is exceptional then $\underline\dim\,X$ lies in $D_{\beta_\ast}(\underline\dim\,Y)$ if and only if $\Hom_{\cD^b}(X,Y)=0=\Ext^1_{\cD^b}(X,Y)$ if and only if $|X|\in \,^\perp Y$.

The following theorem is proved in \cite{IOTW3} in the case $\cA(\beta_\ast)=mod\text-\Lambda$ and $\RR\beta_\ast=\RR^n$.
 
 \begin{thm}[Stability theorem for virtual semi-invariants]\label{Stability theorem for virtual semi-invariants} Let $Y$ be an exceptional module in $\cA(\beta_\ast)$ with $\undim Y=\gamma\in \RR\beta_\ast$. Then, a vector $v\in\RR\beta_\ast$ lies in the convex hull $D_{\beta_\ast}(\gamma)$ of $D_{\ZZ\beta_\ast}(\gamma)$ if and only if the following hold.
  \begin{enumerate}
 \item $\brk{v,\gamma}=0$ and
 \item $\brk{v,\gamma'}\le0$ for all real Schur subroots $\gamma'\subseteq \gamma$ so that $\gamma'\in\NN\beta_\ast$ (these are the dimension vectors of  exceptional submodules $Y'\subseteq Y$ which lie in $\cA(\beta_\ast)$)
 \end{enumerate}
 \end{thm}
 
 Note that the second condition is vacuous when $Y$ is a simple object of $\cA(\beta_\ast)$.
 
 \begin{proof}
 %This is a rewording of the virtual stability theorem proved in \cite{IOTW3}. 
 Let $k$ be the rank of $\cA(\beta_\ast)$. Then we have an isomorphism $\varphi_\ast:\ZZ^k\cong \ZZ\beta_\ast$ given by $\varphi_\ast(a_1,\cdots,a_k)=\sum a_i\beta_i$. This is the linear isomorphism induced by the exact embedding $\varphi:\cA(\beta_\ast)\into mod\text-\Lambda$. Exactness of $\varphi$ implies that $\varphi_\ast$ is an isometry with respect to the form $\brk{\cdot,\cdot}$ and this extends to a linear isometry $\overline\varphi_\ast:\RR^k\cong \RR\beta_\ast$.
 
 Let $\alpha=\varphi_\ast^{-1}(\gamma)\in\NN^k$. Then the Virtual Stability Theorem (\cite{IOTW3}, Theorem 3.1.1) for $\cA(\beta_\ast)$ states, in the present notation, that $\overline\varphi_\ast^{-1}(D_{\beta_\ast}(\gamma))$ is the set of all $x\in\RR^k$ so that
  \begin{enumerate}
 \item[(1)$'$] $\brk{x,\alpha}=0$ and
 \item[(2)$'$] $\brk{x,\alpha'}\le 0$ for all real Schur subroots $\alpha'\subseteq \alpha$
 \end{enumerate}
and $\varphi_\ast^{-1}(D_{\ZZ\beta_\ast}(\gamma))=\overline\varphi_\ast^{-1}(D_{\beta_\ast}(\gamma))\cap \ZZ^k$. Since $\overline\varphi_\ast$ is an isometry, the theorem follows.
 \end{proof}
 
  \begin{cor}\label{cor: comparing D-beta to D-alpha}
  Let $Y\in \cA(\beta_\ast)=|T|^\perp$ with $\underline\dim\,Y=\gamma$ and let $v\in\RR\beta_\ast$. Then $v$ lies in $D_{\beta_\ast}(\gamma)$ if and only if $\underline\dim\,T+\vare v\in D_{\alpha_\ast}(\gamma)$ for all $\vare>0$ sufficiently small.
 \end{cor}

 \begin{proof} ($\then$) Suppose that $v\in D_{\beta_\ast}(\gamma)$.  Then
\begin{enumerate}
\item $\brk{\underline\dim\,T+\vare v,\gamma}=0$ since $Y\in |T|^\perp$.
\item If $\gamma'\subseteq\gamma$ lies in $\cA(\beta_\ast)$ then $\brk{\underline\dim\,T+\vare v,\gamma'}=\vare\brk{v,\gamma'}\le0$.
\item If $\gamma''\subseteq\gamma$ does not lie in $\cA(\beta_\ast)$ then $Y$ has a submodule $W\in\cA(\alpha_\ast)$ of dimension $\gamma''$ so that $W\notin |T|^\perp$. So, either $\Hom_\Lambda(|T|,W)\neq 0$ or $\Ext^1_\Lambda(|T|,W)\neq0$. If $T$ is a module, we cannot have a nonzero homomorphism $T\to W$ since $\Hom_\Lambda(T,Y)=0$. If $T$ is a shifted projective then $\Ext^1_\Lambda(|T|,W)=0$. In either case, we get $\brk{\underline\dim\,T,\gamma''}<0$. Therefore $\brk{\underline\dim\,T+\vare v,\gamma''}<0$ for sufficiently small $\vare$.
\end{enumerate}

($\neht$) Conversely, suppose that $\underline\dim\,T+\vare v\in D_{\alpha_\ast}(\gamma)$ for all $\vare>0$ sufficiently small. Then, $\brk{\underline\dim\,T+\vare v,\gamma}=0$. This implies that $\brk{v,\gamma}=0$ since $\brk{\underline\dim\,T,\gamma}=0$. For any $\gamma'\subseteq \gamma$ where $\gamma'$ is the dimension vector of an object of $\cA(\beta_\ast)=|T|^\perp$, we also have $\brk{\underline\dim\,T,\gamma'}=0$. So, \[
\brk{\underline\dim\,T+\vare v,\gamma'}=\vare\brk{v,\gamma'}\le0
\]
which implies $\brk{v,\gamma'}\le0$.
\end{proof}

 \begin{eg}
 Let $\cA(\alpha_\ast)=\cA(S_1,S_2,S_3)$ be the module category of the quiver $1\ot 2\ot 3$. Let $T$ be the module with $\underline\dim\,T=(0,1,1)^t$. (So, $T=I_2$ is the injective envelope of $S_2$.) Then $T^\perp=\cA(\beta_\ast)=\cA(S_2,P_3)$ is a semi-simple category whose simple objects $S_2$ and $P_3$ are also projective. So, $S_2[1],P_3[1]\in\cC(\beta_\ast)$. Let $Y=P_3$ with dimension vector $\gamma=(1,1,1)^t$ and $v=(0,-1,0)^t=\underline\dim\,S_2[1]$. Then $\brk{v,\gamma}=0$. So, $v\in D_{\beta_\ast}(\gamma)$. The corollary states that
 \[
 	\underline\dim\,T+\vare v=(0,1,1)^t+\vare (0,-1,0)^t=(1,1-\vare,1)^t
 \]
 is an element of $D_{\alpha_\ast}(\gamma)$ for sufficiently small $\vare>0$. In fact, $(1,1-\vare,1)^t\in D_{\alpha_\ast}(\gamma)$ if and only if $\vare\le1$ since $\brk{(1,1-\vare,1)^t,\underline\dim\,S_1}=\vare-1$ is required to be $\le0$ since $S_1$ is a submodule of $T$ in $\cA(\alpha_\ast)$.
 \end{eg}

 Another result that we need, also proved in \cite{IOTW3}, is the virtual generic decomposition theorem. As in the proof of Theorem \ref{Stability theorem for virtual semi-invariants}, this can be reworded as follows.
 
 \begin{thm}[Virtual generic decomposition theorem]
 Suppose that $\{X_1,\cdots,X_k\}$ is a partial cluster tilting set in $\cC(\alpha_\ast)$. Let $P,Q$ be projective objects in $\cA(\alpha_\ast)$ so that $\underline\dim\,Q-\underline\dim\,P=\sum n_i\underline\dim\,X_i$ for positive integers $n_i$. Then for $f$ in an open dense subset of $\Hom_\Lambda(P,Q)$, we have a distinguished triangle in the bounded derived category of $\cA(\alpha_\ast)$:
 \[
 	P\xrarrow fQ\to \coprod n_iX_i\to P[1].
 \]
  \end{thm}
 
\begin{cor}\label{cor: semi-invariants do not cut clusters}
Suppose that $\{X_1,\cdots,X_k\}$ is a partial cluster tilting set in $\cC(\alpha_\ast)$ with dimension vectors $\underline\dim\,X_i=\gamma_i$. Let $Y\in\cA(\alpha_\ast)$ with $\underline\dim\,Y=\gamma$ so that $D_{\alpha_\ast}(\gamma)$ contains $\sum n_i \gamma_i$ where the $n_i$ are positive rational numbers. Then $D_{\alpha_\ast}(\gamma)$ contains $\gamma_i$ for all $i$.
\end{cor}

\begin{proof} By multiplying by a positive integer we may assume that the $n_i$ are all positive integer. For these $n_i$ we take $P,Q$ and $f:P\to Q$ as in the theorem. Then, for general $f$ the semi-invariant $c_Y$ is defined, i.e., $\Hom_\Lambda(f,Y)$ is an isomorphism. By the long exact sequence for the distinguished triangle in the theorem, $\Ext^j_{\cD^b(\cA(\alpha_\ast))}(n_iX_i,Y)=0$ for all $i$ and $j$. So, $X_i\in D_{\alpha_\ast}(\gamma)$ for all $i$.
\end{proof}
 
\subsubsection{$\sigma_T^{-1}$ takes cluster tilting sets to cluster tilting sets}
  Using virtual semi-invariants we will show that the bijection $\sigma_T^{-1}:\cC_T(\alpha_\ast)\to \cC(\beta_\ast)$ takes ext-orthogonal elements to ext-orthogonal elements assuming Properties (a),(b),(c) of Proposition \ref{prop 1.8: Properties of sigma_T} for $k=1$. This will imply that $\sigma_T^{-1}$ takes cluster tilting sets to cluster tilting sets.
  
 Suppose that $X_1,X_2\in \cC_T(\alpha_\ast)$ are ext-orthogonal but $Y_i=\sigma_T^{-1}(X_i)\in\cC(\beta_\ast)$ are not. Then we will obtain a contradiction. 
 
 We have that $\{T,X_1,X_2\}$ is a partial cluster tilting set in $\cA(\alpha_\ast)$. Let $\cA(\gamma_\ast)=|T,X_1,X_2|^\perp$. Then $\cA(\delta_\ast):=\cA(\beta_\ast)\cap\,^\perp\cA(\gamma_\ast)$ is a rank 2 f.g. wide subcategory of $\cA(\alpha_\ast)$. By Properties (a),(b) we have:
 \[
 	|Y_i|^\perp\cap \cA(\beta_\ast)=|X_i|^\perp\cap \cA(\beta_\ast)\supseteq \cA(\gamma_\ast)
 \]
 for each $i$. Therefore, $|Y_i|$ lie in $\,^\perp\cA(\gamma_\ast)=\cA(\delta_\ast)$. Since projectives in $\cA(\beta_\ast)$ are projective in $\cA(\delta_\ast)$ this implies $Y_i\in \cC(\delta_\ast)$.
 
 We are assuming that $Y_i$ are not ext-orthogonal. We can renumber the $Y_i$ so that $Y_1$ is to the left of $Y_2$ in the fundamental domain of $\tau^{-1}[1]$ in the AR-quiver of the bounded derived category of $\cA(\delta_\ast)$. Then $\Hom_{\cD^b}(Y_2,Y_1)=0$ and $\Ext^1_{\cD^b}(Y_1,Y_2)=0$ in $\cD^b=\cD^b(\cA(\delta_\ast)$. If $Y_1,Y_2$ are not ext-orthogonal in the cluster category, we must have $\Ext^1_{\cD^b}(Y_2,Y_1)\neq0$. Also, $Y_1$ must be a module which implies that $\Ext^j_{\cD^b}(Y_2,Y_1)=0$ for $j\neq 0,1$. Therefore, with the notation $\gamma_i=\underline\dim\,Y_i$, we have
 \[
 	\brk{\gamma_2,\gamma_1}=\dim_K\Hom_{\cD^b}(Y_2,Y_1)-\dim_K\Ext^1_{\cD^b}(Y_2,Y_1)<0
 \]
 Also, $\brk{\gamma_1,\gamma_1}>0$. This implies that there are positive rational numbers $a,b$, unique up to scaling, so that $\brk{a\gamma_1+b\gamma_2,\gamma_1}=0$. Let $Z$ be the unique object so that there is an irreducible map $Y_1\to Z$. Then $|Z|\in \,^\perp Y_1$. So, $\brk{\underline\dim\,Z,\gamma_1}=0$. By uniqueness of $a,b$ we have $\underline\dim\,Z=a\gamma_1+b\gamma_2$. Since $|Z|\in \,^\perp Y_1$, this vector $v=a\gamma_1+b\gamma_2$ lies in the support $D_{\beta_\ast}(\gamma_1)$ of the virtual semi-invariant $\sigma_{\gamma_1}$ defined on $\RR\beta_\ast$.
 
By Corollary \ref{cor: comparing D-beta to D-alpha}, $D_{\alpha_\ast}(\gamma_1)$ contains the vector
\[
	\underline\dim\,T+\vare v=\underline\dim\,T+\vare a\gamma_1+\vare b\gamma_2
\]
for $\vare>0$ sufficiently small. By Property (c), this is equal to $c\underline\dim\,T+\vare a\underline\dim\,X_1+\vare b\underline\dim\,X_2$ where $c$ is a number which converges to 1 as $\vare\to0$. By Corollary \ref{cor: semi-invariants do not cut clusters}, the objects $T,X_1,X_2$ lie in $D_{\alpha_\ast}(\gamma_1)$. In other words, $|T|,|X_1|,|X_2|$  lie in $\,^\perp Y_1$. Equivalently, $Y_1$ lies in $|T,X_1,X_2|^\perp$ which is a contradiction. Therefore, $\sigma_T^{-1}$ takes ext-orthogonal elements to ext-orthogonal elements.

\subsubsection{$\sigma_T$ takes cluster tilting sets to cluster tilting sets}

Let $\cK$ be the set of all cluster tilting sets in $\cC(\beta_\ast)$ which are the images under $\sigma_T^{-1}$ of cluster tilting sets in $\cC_T(\alpha_\ast)$. We know that $\cK$ is nonempty since $\cC_T(\alpha_\ast)$ contains at least one cluster tilting set. We claim that $\cK$ is closed under all mutations of cluster tilting sets. Using the well-known fact that all cluster tilting sets over a hereditary algebra are mutation equivalent, this will imply that $\cK$ contains all cluster tilting sets in $\cC(\beta_\ast)$ and that therefore $\sigma_T$ sends all cluster tilting sets in $\cC(\beta_\ast)$ to cluster tilting sets in $\cC_T(\alpha_\ast)$.

To prove the claim, let $Y=\{Y_1,\cdots,Y_\ell\}$ be a cluster tilting set in $\cC(\beta_\ast)$ which lies in $\cK$. Then $X=\{\sigma_TY_1,\cdots,\sigma_TY_\ell,T\}$ is a cluster tilting set in $\cC(\alpha_\ast)$ by definition of $\cK$. For any $j=1,\cdots,\ell$ we want to show that the mutation $\mu_jY$ of $Y$, given by replacing $Y_j$ with $Y_j^\ast\in\cC(\beta_\ast)$ also lies in $\cK$. But this is easy: Take $\mu_jX$. This is the cluster tilting set in $\cC(\alpha_\ast)$ obtained by replacing $\sigma_TY_j$ with the unique other object $Z$ which will complete the cluster tilting set. Since $\sigma_T^{-1}$ takes cluster tilting sets to cluster tilting sets, $\sigma_T^{-1}(\mu_jX)$ is a cluster tilting set in $\cC(\beta_\ast)$. But this is the same as $Y$ except that $Y_j$ is replaced with $\sigma_T^{-1}(Z)\neq Y_j$. This must be equal to $Y_j^\ast$. So, $\mu_jY$ is in $\cK$. So, $\cK$ contains all cluster tilting sets in $\cC(\beta_\ast)$.

This completes the proof of Proposition \ref{prop 1.8: Properties of sigma_T} and Theorem \ref{thm:sigma-T is a bijection} in the case $k=1$.

\subsubsection{Induction step}
 Suppose now that $k=2$ and the proposition and theorem both hold for $k-1$. So, we have $T=\{T_1,\cdots,T_k\}$ a partial cluster tilting set in $\cA(\alpha_\ast)$ and $|T|^\perp=\cA(\beta_\ast)$. By an observation of Schofield, the modules $|T_i|$ can be reordered in such a way that they form an exceptional sequence $(|T_1|,|T_2|,\cdots,|T_k|)$. This implies that $|T_1|,\cdots,|T_{k-1}|$ lie in $|T_k|^\perp$ which we denote $\cA(\gamma_\ast)$. Also, the bijection $\sigma_{T_k}:\cC(\gamma_\ast)\to \cC_{T_k}(\alpha_\ast)$ sends $T_i$ to $T_i$ for $i<k$.
 
 By induction on $k$ we also have a bijection 
 $
 \sigma_{T_\ast}:\cC(\beta_\ast)\to \cC_{T_\ast}(\gamma_\ast)
 $
 given by the partial cluster tilting set $T_\ast=\{T_1,\cdots,T_{k-1}\}$ in $\cC(\gamma_\ast)$.

 \[
 \xymatrixrowsep{15pt}
 \xymatrixcolsep{30pt}
\xymatrix{%begin xy matrix
\cC(\beta_\ast)\ar[r]^{\sigma_{T_\ast}}_\approx &
	\cC_{T_\ast}(\gamma_\ast)\ar[d]_\subseteq\ar[r] &
	 \cC_{T}(\alpha_\ast)\ar[d]_\subseteq\\
&
	\cC(\gamma_\ast) \ar[r]^{\sigma_{T_k}}_\approx&
	\cC_{T_k}(\alpha_\ast)
	}%end xy matrix
 \]
 
 \noi\underline{Claim 1}: The bijection $\sigma_{T_k}$ sends $\cC_{T_\ast}(\gamma_\ast)$ bijectively onto $\cC_T(\alpha_\ast)$ and therefore induces a bijection
 \[
 	\sigma_T:=\sigma_{T_k}\circ\sigma_{T_\ast}:\cC(\beta_\ast)\xrarrow\approx \cC_{T_\ast}(\gamma_\ast)\xrarrow\approx \cC_{T}(\alpha_\ast)
 \]
 
 Proof: An element $Y$ of $\cC(\gamma_\ast)$ lies in $\cC_{T_\ast}(\gamma_\ast)$ iff it is ext-orthogonal to but not equal to $T_i$ for $i<k$. The element $X=\sigma_{T_k}Y\in \cC_{T_k}(\alpha_\ast)$ lies in $\cC_T(\alpha_\ast)$ iff $X$ is ext-orthogonal to but not equal to $T_i$ for $i<k$. Since $\sigma_{T_k}(T_i)=T_i$ and by using the theorem for $k=1$ we see that these conditions are equivalent.
 
 We now show that the bijection $\sigma_T:=\sigma_{T_k}\circ\sigma_{T_\ast}$ satisfies Proposition \ref{prop 1.8: Properties of sigma_T}.
 \begin{enumerate}
 \item[(a)] If $Y\in\cC(\beta_\ast)$ then $\sigma_TY\in \cC_T(\alpha_\ast)$ implies, by definition, that $\{Y,T_1,\cdots,T_k\}$ is a partial cluster tilting set in $\cC(\alpha_\ast)$. So, $\sigma_T$ has Property (a).
 \item[(b)] For any $Y\in\cC(\beta_\ast)$ we have, by induction on $k$, that
 \[
 	\cA(\beta_\ast)\cap|Y|^\perp=\cA(\beta_\ast)\cap|\sigma_{T_\ast}Y|^\perp=\cA(\beta_\ast)\cap|\sigma_{T_k}\sigma_{T_\ast}Y|^\perp
 \]
 Since $\sigma_T=\sigma_{T_k}\circ\sigma_{T_\ast}$, Property (b) holds.
 \item[(c)] By induction on $k$ we have the following for any $Y\in\cC(\beta_\ast)$: 
 \[
 	\underline\dim\,Y+\RR T_\ast=\underline\dim\,\sigma_{T_\ast}Y+\RR T_\ast
 \]
 By the case $k=1$ we have
 \[
 	\underline\dim\,\sigma_{T_\ast}Y+\RR T_k=\underline\dim\,\sigma_{T}Y+\RR T_k
 \]
 Since $\RR T=\RR T_\ast+\RR T_k$, we can put these together to get:
 \begin{equation}\label{eq: Property (c)}
 	\underline\dim\,Y+\RR T=\underline\dim\,\sigma_{T}Y+\RR T
 \end{equation}
 which is equivalent to the statement that $\sigma_T$ satisfies Property (c).
 \end{enumerate}
 
The uniqueness of $\sigma_T$ follows from the following observation.
 
 \noi\underline{Claim 2}: The inclusion map $\RR\beta_\ast\into \RR\alpha_\ast$ induces a linear isomorphism
 \[
 	\lambda_T:\RR\beta_\ast\xrarrow\approx\RR\alpha_\ast/\RR T
 \]
 In other words, $\sigma_TY$ is the unique element of $\cC_T(\alpha_\ast)$ satisfying \eqref{eq: Property (c)}.
 
 Proof: Since $\RR\beta_\ast$ and $\RR T$ have complementary dimensions in $\RR\alpha_\ast$, it suffices to show that they span $\RR\alpha_\ast$. Choose any exceptional sequence in $\cA(\beta_\ast)$, for example the simple objects $(S_\ell,\cdots,S_1)$. Then $(S_\ell,\cdots,S_1,|T_1|,\cdots,|T_k|)$ is an exceptional sequence in $\cA(\alpha_\ast)$. So, the dimension vectors of these modules form a basis for $\RR\alpha_\ast$. Since $S_i\in\cA(\beta_\ast)$, $\underline\dim\,S_i\in\RR\beta_\ast$. Therefore $\RR\beta_\ast+\RR T=\RR\alpha_\ast$, proving Claim 2.\vs2
 
 Property (d) and its converse are easy: $Y_1,Y_2\in \cC(\beta_\ast)$ are ext-orthogonal iff $\sigma_{T_\ast}Y_1,\sigma_{T_\ast}Y_2$ are ext-orthogonal iff $\sigma_{T_k}\sigma_{T_\ast}Y_1,\sigma_{T_k}\sigma_{T_\ast}Y_2$ are ext-orthogonal. Therefore $\sigma_{T}=\sigma_{T_k}\sigma_{T_\ast}$ satisfies Property (d) and both $\sigma_T$ and $\sigma_T^{-1}$ take cluster tilting sets to cluster tilting sets.
 
 This concludes the proof of Proposition \ref{prop 1.8: Properties of sigma_T} and Theorem \ref{thm:sigma-T is a bijection} and therefore also completes the definition of the cluster morphism category.

\section{Signed exceptional sequences}\label{sec 2: signed exceptional sequences}

We are now in a position to explore signed exceptional sequences and prove one of the main theorems of this paper: There is a bijection between signed exceptional sequences and ordered cluster tilting sets.

\subsection{Definition and basic properties}\label{ss 2.1: Def of signed exceptional sequence}

Let $\cA$ be a finitely generated wide subcategory of $mod$-$\Lambda$. Recall Definition \ref{def: signed exceptional sequence}: a \emph{signed exceptional sequence} in $\cA$ is a sequence $(X_1,X_2,\cdots,X_k)$ in $\cA\cup \cA[1]\subset\cD^b(\cA)$ with the following properties.
\begin{enumerate}
\item $(|X_1|,\cdots,|X_k|)$ is an exceptional sequence. So, $|X_i|\in |X_j|^\perp$ for $i<j$.
\item If $X_j=Q[1]$ then $Q$ is a relatively projective object of $|X_{j+1},\cdots,X_k|^\perp$.
\end{enumerate}
In our notation, it is understood that perpendicular categories are taken inside the ambient category $\cA$. Thus $|X_j|^\perp$ means $|X_j|^\perp\cap\cA$.

Let $\cA_j=|X_{j+1},\cdots,X_k|^\perp$. Then (2) is equivalent to the condition: $X_j\in\cC(\cA_j)$. Therefore, a signed exceptional sequence gives a sequence of composable cluster morphisms:
\[
	\cA=\cA_k\xrarrow{[X_k]}\cA_{k-1}\xrarrow{[X_{k-1}]}\cdots \xrarrow{[X_{2}]}\cA_1\xrarrow{[X_{1}]}\cA_0
\]
Conversely, given any composition of cluster morphisms $[Y_1]\circ[Y_2]\circ\cdots\circ[Y_k]:\cA\to\cB$ where each $Y_j$ is a one element cluster tilting set, the sequence $(Y_1,\cdots,Y_k)$ is a signed exceptional sequence in the domain $\cA$.

By the composition law for cluster morphism we have the following.
\begin{prop}\label{prop 2.1: formula for cluster morphism corresponding to signed exceptional sequence}
The cluster morphism corresponding to a signed exceptional sequence $(X_1,\cdots,X_k)$ in $\cA$ is
$
	[X_1]\circ[X_2]\circ\cdots\circ[X_k]=[T(1)]
$
where $T(j)=(T_j, T_{j+1}, \cdots, T_k)$ is the (ordered) partial cluster tilting set in $\cC(\cA)$ given recursively as follows.
\begin{enumerate}
\item $T_k=X_k$.
\item Given $T(j)$, let $T_{j-1}=\sigma_{T(j)}X_{j-1}$.\qed
\end{enumerate}
\end{prop}

We call $T=T(1)=(T_1, \cdots, T_k)$ the ordered partial cluster tilting set in $\cC(\cA)$ corresponding to $(X_1,\cdots,X_k)$.

As an example, consider the sequence of simple modules $(S_1,S_2,\cdots,S_n)$ in admissible order, i.e., so that $(S_n,S_{n-1},\cdots,S_1)$ is an exceptional sequence. Since each $S_k$ is projective in the right perpendicular category of $S_1,\cdots,S_{k-1}$, it can have either sign. So, there are $2^n$ possible signed exceptional sequences coming from this one exceptional sequence. 

\begin{prop}\label{prop: 2 to n distinct clusters}
The cluster tilting sets corresponding to these $2^n$ signed exceptional sequences are all distinct.
\end{prop}

For example, the $2^3=8$ signed exceptional sequences and corresponding cluster tilting sets for the quiver $1\ot 2\ot 3$ are listed in Figure \ref{fig: 8 signed exceptional sequences} where $P_i,I_i,S_i$ are the $i$th projective, injective and simple modules.

\begin{figure}[htbp]
\begin{center}
\[
\begin{array}{cc}
\text{signed exceptional sequence} & \quad\text{ordered cluster tilting set}\quad\\
\hline
\end{array}
\]\[
\begin{array}{cccccccc}
 S_3& S_2& S_1 &&& P_3& P_2& P_1\\
 S_3[1]& S_2& S_1&&& P_3[1]& P_2& P_1\\
S_3[1]& S_2[1]& S_1 &&& P_3[1]& P_2[1]& P_1\\
 S_3[1] & S_2[1]& S_1[1]&&& P_3[1]& P_2[1]& P_1[1]\\
\hline
  S_3[1] & S_2&S_1[1]&&& P_3[1]& S_2& P_1[1]\\
 S_3 & S_2[1]& S_1[1]&&& S_3& P_2[1]&  P_1[1]\\
 S_3 & S_2& S_1[1]&&& I_2& S_2& P_1[1]\\
 S_3& S_2[1]& S_1 &&& S_3& P_2[1]& P_1\\
 \end{array} 
\]
\caption{The $2^n=8$ signed exceptional sequences of simple objects and corresponding ordered cluster tilting sets for the quiver $1\ot 2\ot 3$.}
\label{fig: 8 signed exceptional sequences}
\end{center}
\end{figure}

\begin{proof}
Note that the elements of each of these cluster tilting sets have a natural ordering since $T_n$ is supported at vertex 1, $T_{n-1}$ at vertices 1,2, etc. Suppose that $E_\ast,E_\ast'$ are two signed exceptional sequences whose underlying modules are the simple objects $S_n,\cdots,S_1$. Let $T$, $T'$ be the corresponding cluster tilting sets with their natural ordering. Let $j$ be maximal so that $E_j\neq E_j'$. Then $T_i=T_i'$ for $i>j$ and the support tilting object $T_j'\coprod T_{j+1}\coprod\cdots\coprod T_n$ is the mutation of the support tilting object $T_j\coprod T_{j+1}\coprod \cdots\coprod T_n$ in the $j$-direction. So, $T_j\neq T_j'$ making $T,T'$ nonisomorphic.
\end{proof}

\subsection{First main theorem}\label{ss 2.2: first main theorem}

We can now state and prove the first main theorem. Note that Proposition \ref{prop 2.1: formula for cluster morphism corresponding to signed exceptional sequence} assigns an {ordered cluster tilting set} to each signed exceptional sequence.

\begin{thm}\label{thm 2.3: bijection one}
There is a bijection $\theta_k$ from the set of isomorphism classes of signed exceptional sequences in $\cA$ of length $k$ to the set of {ordered partial cluster tilting sets} in $\cA$ of size $k$ which is uniquely characterized by the following properties.
\begin{enumerate}
\item If $\theta_k(X_1,\cdots,X_k)=T$ then $|X|^\perp=|T|^\perp$. Let $\cB=|T|^\perp$.
\item If $\theta_k(X_1,\cdots,X_k)=T$ then $[T]=[X_1]\circ[X_2]\circ\cdots\circ[X_k]$ as cluster morphisms $\cA\to \cB$.
\item If $\theta_k(X_1,\cdots,X_k)=(T_1,\cdots,T_k)$ then $\theta_{k-j+1}(X_j,\cdots,X_k)=(T_j,\cdots, T_k)$ for all $1\le j\le k$.
\end{enumerate}
\end{thm}

To clarify the wording of the theorem we mean that there is a unique mapping $\theta_k$ satisfying the three listed conditions and that, furthermore, this mapping is a bijection.

\begin{proof}
The formula in Proposition \ref{prop 2.1: formula for cluster morphism corresponding to signed exceptional sequence} gives a function $\theta_k$ satisfying these three conditions. So, it remains to show that $\theta_k$ is uniquely determined and that it is a bijection. We prove both statements at the same time by induction on $k$.

If $k=1$ then Condition (2) implies that $T_1=X_1$. The two sets are both equal to $\cC(\cA)$ and $\theta_1$ must be the identity map.

Suppose $k\ge2$ and $\theta_{k-1}$ is a uniquely determined bijection. Let $(X_1,\cdots,X_k)$ be a signed exceptional sequence. Condition (2) implies that $T=\theta_k(X_1,\cdots,X_k)$ is uniquely determined up to permutation of its elements. But Condition (3) for $j=k-1$ determines the last $k-1$ elements of $T$. So, the first element is also determined. So, the function $\theta_k$ is uniquely determined. 

To show that $\theta_k$ is a bijection, we start with any rigid object $T=\coprod_{1\le j\le k}T_j$ with $k$ summands in a fixed order. This gives a cluster morphism $[T]:\cA\to \cB$. Let $T'=(T_2,T_3,\cdots,T_k)$. This gives a morphism $[T']:\cA\to\cB'$ where $\cB\subset\cB'\subset\cA$. By induction on $k$, there is a unique signed exceptional sequence $Y$ of length $k-1$ so that $\theta_{k-1}(Y)=T'$. Since $T$ is rigid, $T_1$ lies in $\cC_{T'}(\cA)$. By Theorem \ref{thm:sigma-T is a bijection}, there is a unique object $Y_0\in \cC(\cB')$ so that $\sigma_{T'}Y_0=T_1$. The recursive formula in Proposition \ref{prop 2.1: formula for cluster morphism corresponding to signed exceptional sequence} then gives $\theta_k(Y_0,Y)=T$.
\end{proof}

\begin{rem}
Theorem \ref{thm 2.3: bijection one} has been extended to the $m$-cluster category by the first author and to arbitrary finite dimensional algebras using $\tau$-tilting by Buan and Marsh. Details will appear when available.
\end{rem}

The bijection between {ordered cluster tilting set}s and signed exceptional sequences can be used to define the composition of cluster morphisms.

\begin{cor}
If $[T]:\cA_0\to \cA_1$ and $[T']:\cA_1\to \cA_2$ are cluster morphism, the composition $[T']\circ[T]:\cA_0\to \cA_2$ can be given as follows. Take two signed exceptional sequences $(X_1,\cdots,X_\ell)$ in $\cA_1$ and $(Y_1,\cdots,Y_k)$ in $\cA_0$ so that $\theta_\ell(Y)=T$ in some order and $\theta_k(X)=T'$ in some order. Then $[T']\circ [T]=\theta_{k+\ell}^{-1}(X,Y)$.
\end{cor}

\begin{proof}
This follows immediately from Property (2) in Theorem \ref{thm 2.3: bijection one}.
\end{proof}

The inverse bijection $\theta_k^{-1}$ from {ordered cluster tilting set}s to signed exceptional sequences is given by the following ``twist'' formula which is based on \cite{SM}.

A finite set of vectors in $\QQ^n$ will be called \emph{nondegenerate} if it is linearly independent and satisfies the condition that the Euler-Ringel pairing $\brk{\cdot,\cdot}$ is nondegenerate on the span of any subset of the set of vectors.

\begin{defn}\label{def: twist equation}
 We define the \emph{right twist} of any nondegenerate sequence of vectors $v_\ast=(v_1,\cdots,v_k)$ in $\QQ^n$ to be the unique sequence of vectors $\tau_+(v_\ast)=(w_1,\cdots,w_k)$ satisfying the following.
\begin{enumerate}
\item For each $j$, $w_j-v_j$ is a linear combination of $v_i$ for $i>j$.
\item $\brk{v_i,w_j}=0$ for all $i> j$.
\end{enumerate}
Note that, given (1), Condition (2) is equivalent to 

($2'$) $\brk{w_i,w_j}=0$ for all $i> j$.

\noi We say that $(w_\ast)$ is an \emph{integer right twist} of $(v_\ast)$ if each $w_j$ is an integer linear combination of the $v_i$.
\end{defn}

\begin{prop}\label{prop: signed exceptional sequences are nondegenerate}
The dimension vectors of any signed exceptional sequence $(X_1,\cdots,X_k)$ is nondegenerate with respect to the pairing $\brk{\cdot,\cdot}$. Furthermore $\tau_+(\undim X_\ast)=(\undim X_\ast)$.
\end{prop}

\begin{proof}
Any subset of the $X_i$ forms an exceptional sequence. So the span of their dimension vectors is the span of the dimension vectors of a wide subcategory which is equivalent to the module category of a finite acyclic quiver. Thus $\brk{\cdot,\cdot}$ is nondegenerate on any such span. The equation $\tau_+(\undim X_\ast)=(\undim X_\ast)$ follows from Proposition \ref{prop 2.1: formula for cluster morphism corresponding to signed exceptional sequence} and the properties of $\sigma_T$ listed in Proposition \ref{prop 1.8: Properties of sigma_T}, in particular (c).
\end{proof}

We also need the following important theorem essentially due to Schofield.
 
 \begin{thm}\label{thm: Schofield's observation}
 Any partial cluster tilting set $\{T_1,\cdots,T_k\}$ giving a morphism $[T]:\cA\to\cB$ can be ordered in such a way that it forms a signed exceptional sequence.
 \end{thm}
 
 \begin{proof}
 Schofield \cite{S92} proved this in the case when the $T_i$ lie in $\cA$. But this case extends easily to cluster tilting sets by putting the shifted projective objects last.
 \end{proof}
 
 Theorem \ref{thm: Schofield's observation} and Proposition \ref{prop: signed exceptional sequences are nondegenerate} imply that the set of dimension vectors of any partial cluster tilting set is nondegenerate, therefore, its right twist is defined.

\begin{thm}\label{thm: formula for theta inverse}
The sequence of dimension vectors of any ordered partial cluster tilting set $T=(T_1,\cdots,T_k)$ has an integer right twist $\tau_+(\undim T_i)=(\undim X_i)$ which gives the dimension vectors of the corresponding signed exceptional sequence $(X_1,\cdots,X_k)=\theta_k^{-1}(T)$.
\end{thm}

\begin{proof}
Let $T_{>j}=(T_{j+1},\cdots,T_k)$. Then it follows from the formula $T_j=\sigma_{T_{>j}}X_j$ that $\undim T_j-\undim X_j$ is an integer linear combination of the vectors $\undim T_i$ for $i>j$. By downward induction on $j$, this implies that the span of $\undim X_i$ for $i>j$ is equal to the span of the vectors $\undim T_i$ for $i>j$. So, the fact that $(X_i)$ is a signed exceptional sequence implies that $
	\brk{\undim T_i,\undim X_j}=0
$ for $i>j$. Therefore, the sequence of dimension vectors $(\undim X_i)$ satisfies the definition of an integral right twist for $(\undim T_i)$.
\end{proof}

\subsection{Permutation of signed exceptional sequences}\label{ss 2.3: permutation of signed exc seq}

The question we address here is: When can the terms in a signed exceptional sequence be permuted? Without the signs, the answer is given by the following trivial observation.

\begin{prop}\label{permutation of exc seqs}
Suppose that $(M_1,\cdots,M_n)$ is an exceptional sequence in $mod\text-\Lambda$. Let $\sigma$ be any permutation of $n$. Then $(M_{\sigma(1)},\cdots,M_{\sigma(n)})$ is an exceptional sequence if and only if $M_i,M_j$ are hom-ext perpendicular whenever $i<j$ and $\sigma(i)>\sigma(j)$.\qed
\end{prop}

 We will show that the same holds for signed exceptional sequences. This is not completely obvious since there is a condition on which modules can be shifted. We consider the signed version of Proposition \ref{permutation of exc seqs} in the key case when $n=2$ and $\sigma$ is a transposition.

\begin{lem}\label{n=2 case of commuting sig exc seq}
Suppose that $(X,Y)$ is a signed exceptional sequence in $\cA$ with corresponding ordered partial cluster tilting set $(Z,Y)$ in $\cC(\cA)$ where $Z=\sigma_YX$. Then the following are equivalent.
\begin{enumerate}
\item $(Y,X)$ is a signed exceptional sequence in $\cA$.
\item $(|Y|,|X|)$ is an exceptional sequence.
\item $|X|,|Y|$ are hom-ext orthogonal.
\item $Z,Y$ are hom orthogonal and $Z=X$.
\end{enumerate}
Furthermore, when this holds, $(Y,X)$ is the ordered partial cluster tilting set corresponding to the signed exceptional sequence $(Y,X)$. I.e., $\sigma_XY=Y$.
\end{lem}

 \begin{proof}
 It follows from the definitions that (1) implies (2) and that (2), (3) are equivalent.
 
 $(3)\then (4)$. By Property (c) of $\sigma_Y$ we know that $\undim Z=\undim X+c\undim Y$. Then
 \[
 	\brk{\undim Y,\undim Z}=\brk{\undim Y,\undim X}+c\brk{\undim Y,\undim Y} =c\dim\End_\Lambda(Y)
 \]
 \[
 	\brk{\undim Z,\undim Y}=\brk{\undim X,\undim Y}+c\brk{\undim Y,\undim Y}=c\dim\End_\Lambda(Y)
 \]
 By Schofield's theorem above, one of these must be zero. So, $c=0$ and $Z=X$ since $Z$ is uniquely determined by its dimension vector. So, $|Z|,|Y|$ are hom-ext orthogonal. This implies (4) when $Y,Z$ have the same sign. So, it is left to consider the case when one of them, say $Y$, is a module and $X=Z=P[1]$ where $P$ is projective. Then $\Hom_{\cD^b(\cA)}(P[1],Y)=0$ and $\Hom_{\cD^b(\cA)}(Y,P[1])=\Ext^1(Y,P)=0$ by (3). So, (4) holds.
 
 $(4)\then(3)$. If $Z=X$ and $Y$ have the same sign, this is clear. So, suppose they have opposite signs. Say, $Y$ is a module and $Z=X=P[1]$. Since $Y,Z$ form a partial cluster tilting set we have $\Hom_\Lambda(P,Y)=0$. Also $\Ext^1_\Lambda(P,Y)=0$ since $P$ is projective. So, $Y\in |X|^\perp$. Since $(X,Y)$ is given to be a signed exceptional sequence, we also have $|X|\in Y^\perp$. So, $|X|,|Y|$ are hom-ext orthogonal.
 
$ (2),(4)\then (1)$. Given that $(|Y|,|X|)$ is an exceptional sequence, we just need to check that the signs on $Y,X$ are admissible. But, by (4), $X,Y$ are either objects of $\cA$ or shifted projective objects. So, by definition, their signs are admissible and $(Y,X)$ is a signed exceptional sequence.

Finally, the last statement $\sigma_XY=Y$ follows from Property (e) of the function $\sigma_X$.
 \end{proof}

 \begin{lem}\label{transpositions of sig exc seq}
 Let $(X_1,\cdots,X_n)$ be a signed exceptional sequence with corresponding {ordered cluster tilting set} $(T_1,\cdots,T_n)$. Then, for each $i$, the following are equivalent.
 \begin{enumerate}
  \item  $(X_1,\cdots,X_{i-1},X_{i+1},X_i,X_{i+2},\cdots,X_n)$ is a signed exceptional sequence.
 \item $(|X_1|,\cdots,|X_{i-1}|,|X_{i+1}|,|X_i|,|X_{i+2}|,\cdots,|X_n|)$ is an exceptional sequence.
 \item $|X_i|,|X_{i+1}|$ are hom-ext perpendicular.
 \end{enumerate}
 Furthermore, when these hold, the signed exceptional sequence in \emph{(1)} corresponds to the {ordered cluster tilting set} $(T_1,\cdots,T_{i-1},T_{i+1},T_i,T_{i+2},\cdots,T_n)$.
 \end{lem}
 
 \begin{proof} The equivalence of (1), (2) and (3) follows from Lemma \ref{n=2 case of commuting sig exc seq} applied to the signed exceptional sequence $(X_i,X_{i+1})$ in $\cA=|X_{i+2},\cdots,X_n|^\perp$.
 
 To prove the last statement, we use the fact, also proved in Lemma \ref{n=2 case of commuting sig exc seq}, that $\sigma_{X_i}X_{i+1}=X_{i+1}$ and $\sigma_{X_{i+1}}X_{i}=X_{i}$. Then $T_{i+1}=\sigma_{T'}X_{i+1}$ where $T'=(T_{i+2},\cdots,T_n)$ and, using the notation $T''=(T_{i+1},T_{i+2},\cdots,T_n)$, we also have
 \[
 	T_i=\sigma_{T''}X_i =\sigma_{T'}\sigma_{X_{i+1}}X_i=\sigma_{T'}X_i\,.
 \]
 If $(M_1,\cdots,M_n)$ is the {ordered cluster tilting set} associated to the signed exceptional sequence in (1) then we must have $M_j=T_j$ for $j\neq i,i+1$ and
 \[
 	M_{i+1}=\sigma_{T'}X_i=T_i
 \]
 \[
 	M_i=\sigma_{T'}\sigma_{X_i}X_{i+1}=\sigma_{T'}X_{i+1}=T_{i+1}
 \]
 proving the last claim of the lemma.
  \end{proof}

Given any permutation $\sigma$ of $n$, the \emph{inversions} of $\sigma$ are defined to be pairs of integers $(i,j)$ so that $i<j$ and $\sigma(i)>\sigma(j)$.

\begin{prop}\label{prop: permutation of sig exc seqs}
Suppose that $(M_1,\cdots,M_n)$ is a signed exceptional sequence in $mod\text-\Lambda$ with corresponding {ordered cluster tilting set} $(T_1,\cdots,T_n)$. Let $\sigma$ be any permutation of $n$. Then the following are equivalent.
\begin{enumerate}
\item $(M_{\sigma(1)},\cdots,M_{\sigma(n)})$ is a signed exceptional sequence.
\item $(|M_{\sigma(1)}|,\cdots,|M_{\sigma(n)}|)$ is a signed exceptional sequence.
\item $|M_i|,|M_j|$ are hom-ext orthogonal for all inversions $(i,j)$ of $\sigma$.
\end{enumerate}
When this holds, the {ordered cluster tilting set} corresponding to $(M_{\sigma(i)})$ is $(T_{\sigma(1)},\cdots,T_{\sigma(n)})$.
\end{prop}

\begin{proof} If $\sigma$ has only one inversion then it is a simple transposition $(i,i+1)$ and the proposition follows from Lemma \ref {transpositions of sig exc seq} in that case. So, suppose $\sigma$ has $k\ge2$ inversions and the proposition holds for $k-1$. Then $\sigma$ is the product of $k$ simple transpositions: $\sigma=\tau_1\tau_2\cdots\tau_k$. Let $\sigma'=\tau_1\cdots\tau_{k-1}$. Then it is an elementary fact that $\sigma'$ has $k-1$ inversions each of which is an inversion of $\sigma$. We can now prove the proposition for $k$.

If $(M_i),(M_{\sigma(i)})$ are signed exceptional sequences then $(|M_i|), (|M_{\sigma(i)}|)$ are exceptional sequences. This implies (3). Conversely, suppose that $|M_i|,|M_j|$ are hom-ext orthogonal for every inversion $(i,j)$ of $\sigma$. Then, a fortiori, the same holds for every inversion $(i,j)$ of $\sigma'$. By induction on $k$ we have that $(M_{\sigma'(1)},\cdots,M_{\sigma'(n)})$ is a signed exceptional sequence with corresponding {ordered cluster tilting set}  $(T_{\sigma'(1)},\cdots,T_{\sigma'(n)})$. By Lemma \ref{transpositions of sig exc seq} we can apply the last simple transposition $\tau_k$ to show that $(M_{\sigma(i)})$ is a signed exceptional sequence with corresponding {ordered cluster tilting set} $(T_{\sigma(i)})$.
\end{proof}

\subsection{{\it c} -vectors}\label{ss 2.4: c-vectors}

In lieu of the definition, we first recall the following characterizing property of $c$-vectors associated to a cluster tilting set. (See \cite{IOTW3}, \cite{ST}, \cite{IOs}, \cite{IOTW2a}.) Since there are two notions of correspondence, we use the term \emph{exchange correspondence} for this association.

\begin{thm}\label{thm: characterization of c-vectors}
Given an ordered cluster tilting set $T=(T_1,\cdots,T_n)$ for $mod\text-\Lambda$, the exchange-corresponding $c$-vectors are real Schur roots $\beta_1,\cdots,\beta_n$ which are uniquely determined by the following equation:
\begin{equation}\label{eq characterizing c-vectors}
	\brk{\undim T_i,\beta_j}=-f_i\delta_{ij}
\end{equation}
where $f_i=\dim_K\End_\Lambda(T_i)$.
\end{thm}

It follows immediately that the set of $c$-vectors $\beta_i$ determines the cluster tilting set $T$.

In \cite{ST}, Speyer and Thomas gave a characterization of $c$-vectors. In terms of signed exceptional sequences their theorem can be phrased as follows. 

\begin{thm}\cite{ST}\label{ST: c vectors are exceptional sequences}
A set $\{\beta_1,\cdots,\beta_n\}$ of real Schur roots is the set of $c$-vectors of a cluster tilting set if and only if there is a signed exceptional sequence $X_1,\cdots,X_n$ with $\undim X_i=-\beta_{\sigma(i)}$ for some permutation $\sigma$ so that $X_1,\cdots,X_k$ are hom-orthogonal shifted modules and $X_{k+1},\cdots,X_n$ are hom-orthogonal modules.
\end{thm}

The next theorem shows that, under certain conditions, the bijection between ordered cluster tilting sets and signed exceptional sequences is equivalent to the exchange correspondence between cluster tilting sets and $c$-vectors. It is not immediate how Theorem \ref{thm: Which sig exc seqs are c-vectors?} and Theorem \ref{ST: c vectors are exceptional sequences} are related.

\begin{thm}[Exchange-correspondence=bijective correspondence]\label{thm: Which sig exc seqs are c-vectors?}
Given any signed exceptional sequence $(X_1,\cdots,X_n)$, the negatives of the dimension vectors $\gamma_i=\undim X_i$ form the set of $c$-vectors for some cluster tilting set if and only if the ordered cluster tilting set $(T_1,\cdots,T_n)$ bijectively corresponding to $(X_i)$ has the property that 
\begin{equation}\label{eq: good order for cluster}
\Hom_\Lambda(|T_i|,|T_j|)=0=\Ext^1_\Lambda(|T_i|,|T_j|)
\end{equation}
for all $i<j$. Furthermore, $(-\gamma_1,\cdots,-\gamma_n)$ is equal to the ordered set of $c$-vectors exchange-corresponding to the ordered cluster tilting set $(T_i)$.
\end{thm}

\begin{proof} Suppose that $(T_1,\cdots,T_n)$ is an ordered cluster tilting set satisfying \eqref{eq: good order for cluster} and let $(X_1,\cdots,X_n)$ be the corresponding signed exceptional sequence. Then we will show that $(-\gamma_i=-\undim X_i)$ satisfies \eqref {eq characterizing c-vectors} and are thus the $c$-vectors of the cluster tilting set.

We will first find the solution of the equations \eqref {eq characterizing c-vectors}. Condition \eqref{eq: good order for cluster} implies that $a_{ij}:=\brk{\undim T_i,\undim T_j}=0$ if $i<j$. We also have $a_{ii}=\dim \End T_i=f_i$. By elementary linear algebra this implies that there is a unipotent lower triangular matrix $(b_{jk})$ so that
\[
	\brk{
	\undim T_i,\sum_j b_{jk}\undim T_j
	}=\sum_j\brk{\undim T_i,\undim T_j}b_{jk}=\sum_j a_{ij}b_{jk}=f_i\delta_{ik}
\]
Therefore, $-\beta_k=-\sum_j b_{jk}\undim T_j$ are the $c$-vectors of the cluster tilting set. \vs2

\underline{Claim}: $\beta_k=\gamma_k=\undim X_k$ for each $k$.
\vs2
Proof: By Theorem \ref{thm: formula for theta inverse}, $\undim X_j-\undim T_j$ is a linear combination of $\undim T_i$ for $i>j$. If we let $k$ be maximal so that $\beta_k\neq\undim X_k$ then this tells us that $\beta_k-\undim X_k$ is a linear combination of $\undim T_i$ for $i>k$, say,
\[
	\beta_k-\undim X_k=\sum a_i\undim T_i\neq 0.
\]
Let $j$ be minimal so that $a_j\neq0$. Then
\[
	\brk{\undim T_j,\beta_k-\undim X_k}=\sum a_i\brk{\undim T_j,\undim T_i}=a_jf_j\neq0.
\]
But this is impossible since $\brk{\undim T_j,\beta_k}=0$ by construction of $\beta_k$ and $\brk{\undim T_j,\undim X_k}=0$ since $|X_k|\in |T_j|^\perp$.\vs2

Conversely, given that $-\gamma_i=-\undim X_i$ are the $c$-vectors of an ordered cluster tilting set $T'=(T_1',\cdots,T_n')$ we will show that $T'=T$ and that the cluster tilting set satisfies \eqref{eq: good order for cluster}.

Using Theorem \ref{thm: Schofield's observation}, there exists a permutation $\sigma$ of $n$ so that $\Hom_\Lambda(|T'_{\sigma(i)}|,|T'_{\sigma(j)}|)=0=\Ext^1_\Lambda(|T'_{\sigma(i)}|,|T'_{\sigma(j)}|)$ for $i<j$. By what we have shown in the first part of this proof, this implies that the signed exceptional sequence $(X_{\sigma(i)})$ corresponding to $(T'_{\sigma(i)})$ has negative dimension vectors equal to the ordered set of $c$-vectors $-\gamma_{\sigma(i)}=-\undim X_{\sigma(i)}$. Since $X$ and $(X_{\sigma(i)})$ are both signed exceptional sequences, we can apply Proposition \ref{prop: permutation of sig exc seqs} to conclude that $T'$ is the ordered cluster tilting set corresponding to $X$. In other words, $T'=T$ as claimed. This proved all the statements of the theorem.
\end{proof}

For example, in Figure \ref{fig: 8 signed exceptional sequences}, the top 4 {ordered cluster tilting set}s satisfy \eqref{eq: good order for cluster}. So, the dimension vectors of the corresponding signed exceptional sequences satisfy \eqref{eq characterizing c-vectors} and are thus the negatives of the $c$-vectors corresponding to the cluster tilting set. Also, the top 4 signed exceptional sequence in Figure \ref{fig: 8 signed exceptional sequences} satisfy the criteria of Theorem \ref{ST: c vectors are exceptional sequences}.

Since the objects in a cluster tilting set are ext-orthogonal, it is easy to see that condition \eqref{eq: good order for cluster} is equivalent to the condition
\begin{equation}\label{eq: good order for cluster B}
	\brk{\undim T_i,\undim T_j}=0.
\end{equation}

By Schofield's observation (Theorem \ref {thm: Schofield's observation}), we get the following corollary.

\begin{cor}
Let $(T_1,\cdots,T_n)$ be an ordered cluster tilting set with corresponding ordered set of $c$-vectors $(-\gamma_1,\cdots,-\gamma_n)$. Then there exists a permutation $\sigma$ so that $(\gamma_{\sigma(1)},\cdots,\gamma_{\sigma(n)})$ are the dimension vectors of a signed exceptional sequence. Furthermore, $\sigma$ has this property if and only if 
\begin{equation}\label{eq: good permutation order for cluster}
\brk{\undim T_{\sigma(i)},\undim T_{\sigma(j)}}=0
\end{equation}
for all $i<j$.
\end{cor}

\begin{proof} The existence of $\sigma$ satisfying \eqref{eq: good permutation order for cluster} follows from the observation of Schofield. By Theorem \ref{thm: Which sig exc seqs are c-vectors?} this implies that $(\gamma_{\sigma(i)})$ are the dimension vectors of the signed exceptional sequence corresponding to the ordered cluster tilting set $(T_{\sigma(i)})$.

Conversely, suppose that $\sigma$ is a permutation of $n$ so that $(\gamma_{\sigma(i)})$ are the dimension vectors of a signed exceptional sequence. Let $(M_{\sigma(1)},\cdots,M_{\sigma(n)})$ be the corresponding {ordered cluster tilting set}. By Theorem \ref{thm: Which sig exc seqs are c-vectors?}, this cluster tilting set has the property that $\brk{\undim M_{\sigma(i)},\undim M_{\sigma(j)}}=0$ for $i<j$ and $(-\gamma_{\sigma(i)})$ is the corresponding ordered set of $c$-vectors. Since ordered cluster tilting sets are determined by their ordered set of $c$-vectors, this implies that $M_{\sigma(i)}=T_{\sigma(i)}$ for all $i$ proving the second half of the corollary.\end{proof}

\begin{rem} Using Theorems \ref{thm: formula for theta inverse} and \ref{thm: Which sig exc seqs are c-vectors?}, this corollary gives another method to find the $c$-vectors of a cluster tilting set $(T_1,\cdots,T_n)$: First find $\sigma$ satisfying \eqref{eq: good permutation order for cluster}. Then
\[
	(-\gamma_{\sigma(i)})=\tau_+(\undim T_{\sigma(i)}).
\]
\end{rem}

%\newpage
%%%%%%%%%%%%%%%%%%%%%%%%%%
%
%                Section  {Geometric realization of the cluster morphism category}
%
%%%%%%%%%%%%%%%%%%%%%%%%%%

\section{Classifying space of the cluster morphism category}\label{sec 3: classifying space of G(S)}

In this section we state the second main theorem of this paper, give an extension of this theorem more suitable for induction, give an outline and verify all the steps in the outline with some review of basic topics such as Quillen's Theorem A.

\subsection{Statement of the theorem}\label{ss 3.1: statement of theorem}

{Here is the second main theorem.}

\begin{thm}\label{thm 3.1: 2nd main theorem}
The classifying space of the cluster morphism category of any hereditary algebra of finite representation type is a $K(\pi,1)$ where $\pi$ is the picture group of the algebra as defined in \cite{IOTW4}.
\end{thm}

The fundamental group of the cluster morphism category is described below together with a generalization of this theorem to extension closed full subcategories of the module category. This generalization is easier to prove since we can apply induction on the number of roots in the extension closed subset.

Recall that, for any pair of real Schur roots $\alpha,\beta$, $hom(\alpha,\beta)=\dim\Hom_\Lambda(M_\alpha,M_\beta)$ and $ext(\alpha,\beta)=\dim\Ext^1(M_\alpha,M_\beta)$. We say that $\alpha,\beta$ are \emph{hom-orthogonal} if $M_\alpha,M_\beta$ are hom-orthogonal.

\begin{defn}\label{def: convex set of roots}
A set $\cS$ of real Schur roots of $mod$-$\Lambda$ will be called \emph{convex} if it satisfies the following two conditions.
\begin{enumerate}
\item Given any wide subcategory $\cA(\alpha_\ast)$ of $mod$-$\Lambda$ whose simple objects have dimension vectors $\alpha_i\in \cS$, the set $ab(\alpha_\ast)$ of all dimension vectors of all exceptional modules in $\cA(\alpha_\ast)$ is a finite subset of $\cS$.
\item There is a partial ordering of $\cS$ so that for all $\alpha,\beta\in\cS$ with $\alpha<\beta$ we have $hom(\beta,\alpha)=0=ext(\alpha,\beta)$. 
\end{enumerate}
\end{defn}

For example, in $A_3$ with straight orientation, $\cS=\{\alpha,\beta\}$ with $\alpha=(1,1,0)^t$, $\beta= (0,1,1)^t$ satisfies (1) since its two elements are not hom-orthogonal. So, the elements of $ab(\alpha,\beta)$ are not required to be in $\cS$. This is possible since the middle term of the extension is not indecomposable. The partial ordering is $\alpha<\beta$.

If $\Lambda$ is of finite representation type then all roots are real Schur roots and the set of all roots is convex. The set of all preprojective (or preinjective) roots, i.e., the dimension vectors of the projective modules in $mod$-$\Lambda$ is also convex. We note that, in Definition \ref{def: convex set of roots}, the simple objects of $\cA(\alpha_\ast)$ are not necessarily simple in $mod\text-\Lambda$.

\begin{defn}\label{def: G(S) for S convex}
If $\cS$ is any convex set of real Schur roots, let $G(\cS)$ be the groups given with generators and relations as follow.
\begin{enumerate}
\item $G(\cS)$ has one generator $x(\beta)$ for every $\beta\in\cS$.
\item For each pair $(\alpha,\beta)$ of hom-orthogonal roots in $\cS$ so that $ext(\alpha,\beta)=0$, we have the relation:
\[
	x(\alpha)x(\beta)=\prod x(a_i\alpha+b_i\beta)
\]
where the product is over all $a_i\alpha+b_i\beta\in ab(\alpha,\beta)$ in order of the ratio $a_i/b_i$.\end{enumerate}
\end{defn}

When $\cS$ is the set of all positive roots for a Dykin quiver, $G(\cS)$ is the \emph{picture group} of the quiver as defined in \cite{IOTW4}.

We observe that the order of objects in the product $\prod x(\gamma_i)$ is the right to left order (``backwards'' order) of the objects $M_{\gamma_i}$ in the AR quiver of $\cA(\alpha,\beta)$. For example, in the case $B_2$, the modulated quiver $\RR\leftarrow\CC$ with simple roots $\alpha=(1,0)^t$ and $\beta=(0,1)^t$, the AR quiver is:
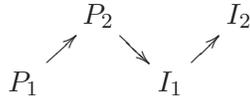
\[
\xymatrixrowsep{10pt}\xymatrixcolsep{10pt}
\xymatrix{%begin xy matrix
& P_2\ar[dr] && I_2\\
P_1\ar[ru]&& 
	I_1\ar[ru]
	}%end xy matrix
\]
These modules have dimension vectors $\underline\dim\,P_1,\underline\dim\,P_2,\underline\dim\,I_1,\underline\dim\,I_2=\alpha,2\alpha+\beta,\alpha+\beta,\beta$. The ratios $a_i/b_i$ for these modules are: $\infty,2,1,0$ respectively. So, the order is reversed in the product and we get:
\[
	x(\alpha)x(\beta)=x(\beta)x(\alpha+\beta)x(2\alpha+\beta)x(\alpha)
\]
or: $[x(\alpha),x(\beta)]=x(\alpha+\beta)x(2\alpha+\beta)$ where we always use the notation:
\[
	[x,y]:=y^{-1}xyx^{-1}
\]

\begin{defn}
If $\cS$ is any convex set of real Schur roots, let $\cG(\cS)$ be the full subcategory of the cluster morphism category whose objects are all $\cA(\alpha_\ast)$ where $\alpha_\ast\subseteq\cS$ is a finite set of hom-orthogonal roots which form an exceptional sequence. (By definition of convexity this implies that the dimension vector of every exceptional object in $\cA(\alpha_\ast)$ lies in $\cS$.)
\end{defn}

Note that $\cG(\cS)$ always has at least one object $\cA(\emptyset)$. In the classifying space $B\cG(\cS)$, we use this as the base point. The choice of base point is important in order to make the fundamental group of $B\cG(\cS)$ well-defined.

\begin{thm}\label{thm 3.5: G(S) is K(pi,1)}
Let $\cS$ be any finite convex set of real Schur roots. Then the classifying space of the cluster morphism category $\cG(\cS)$ is a $K(\pi,1)$ with $\pi=\pi_1\cG(\cS)=G(\cS)$:
\[
	B\cG(\cS)\simeq BG(\cS)=K(G(\cS),1).
\]
\end{thm}

\subsection{HNN extensions and outline of proof}\label{ss 3.2: outline of G(S)}

The proof of Theorem \ref{thm 3.5: G(S) is K(pi,1)} will be by induction on $|\cS|$. If $\cS$ is empty, then $\cG(\emptyset)$ has only one object $\cA(\emptyset)$ and one morphism: the identity map on this object. The classifying space is therefore a single point which is $K(\pi,1)$ with $\pi=\{e\}$, the trivial group. So, the theorem holds in this case.

If $\cS$ is nonempty we will construct two convex proper subsets $\cS_\omega\subseteq \cS_0\subset \cS$ (in \eqref{eq: def of S-omega} and Lemma \ref{lem: construction of S-0} below). Then, by induction on $|\cS|$, the classifying space $B\cG(\cS_0)$, $B\cG(\cS_\omega)$ will be a $K(\pi,1)$'s with $\pi=G(\cS_0)$, $G(\cS_\omega)$, respectively. We will show that $B\cG(\cS)$ can be obtained from $B\cG(\cS_0)$, $B\cG(\cS_\omega)$ in the following steps.

First we show (Lemma \ref{lem: G(S) is G+ cup G-}) that $\cG(\cS)$ is the union of two subcategories $\cG_+,\cG_-$ so that
\[
	B\cG(\cS)=B\cG_+\cup B\cG_-
\]
and
\[
	B\cG_+\cap B\cG_-=B\cG(\cS_\omega)\smallcoprod B\cH(\cS,\omega)
\]
where, by Proposition \ref{prop: isomorphism H=G(S-w)}, there is an isomorphism 
\[
	\f:\cH(\cS,\omega)\xrarrow\cong \cG(\cS_\omega)
\]
Next, we show (Lemma \ref{lem: key lemma}) that there is a homotopy equivalence
\[
	B\cG_+\simeq B\cG(\cS_0) 
\]
and (Lemma \ref{lem: BG- is a cylinder}) a homeomorphism
\[
	B\cG_-\cong B\cH(\cS,\omega)\times[0,1]\cong B\cG(\cS_\omega)\times[0,1]
\]
So,
\[
	B\cG(\cS)=B\cG_+\cup B\cG_-=B\cG_+\cup B\cH(\cS,\omega) \times[0,1]
\]
We also show in Lemma \ref{lem: BG- is a cylinder} that the cylinder $B\cH(\cS,\omega)\times [0,1]$ is attached to $B\cG_+$ on its two ends by mappings
\[
	B\f_i:B\cH(\cS,\omega)\cong B\cG(\cS_\omega)\to B\cG_+
\]
for $i=0,1$ induced by functors $\f_i:\cH(\cS,\omega)\to \cG_+$ where $\f_0:\cH(\cS,\omega)\into \cG_+$ is the inclusion functor and $\f_1$ is the composition of $\f:\cH(\cS,\omega)\cong \cG(\cS_\omega)$ with the inclusion $\cG(\cS_\omega)\into \cG_+$. 

Next, we show that the induced maps on fundamental groups
\[
	\pi_1(\f_i):G(\cS_\omega)\into G(\cS_0)
\]
are monomorphisms where $\pi_1(\f_1)=\f$ and $\pi_1(\f_0)=\psi$ in the notation below. This is shown in Proposition \ref{prop: G-omega to GS is split mono} for $\pi_1(\f_1)$ and Proposition \ref{prop: psi has left inverse} for $\pi_1(\f_0)=\psi$.

This will be enough to prove Theorem \ref{thm 3.5: G(S) is K(pi,1)} because of the following well-known result about HNN extensions.

\begin{defn}
An \emph{HNN extension} of a group $G$ is given by a subgroup $H$ which is embedded in $G$ in two different ways. Let $\f,\psi:H\to G$ be two such group monomorphisms. Then $N(H,G,\f,\psi)$ is the quotient of the free product $G\ast \brk{t}$ of $G$ with the free group on one generator $t$ modulo the relation
\[
	t\f(h)=\psi(h)t
\]
for every $h\in H$.
\end{defn}

Given $G,H,\f,\psi$ suppose that $BG=K(G,1)$, $BH=K(H,1)$ and $f,g:BH\to BG$ are continuous maps so that
\begin{enumerate}
\item $f$ is pointed (takes basepoint to basepoint) and induces the group homomorphism $\pi_1(f)=\f:H\into G$ and
\item $g$ is not pointed but there is a path $\gamma$ from $g(\ast)$ to the basepoint of $BG$ so that the induced homomorphism on $\pi_1$ is
\[
	\pi_1(g,\gamma)=\psi:H\into G
\]
Here $\pi_1(g,\gamma)$ sends $[\alpha]\in \pi_1BH=H$, represented by the loop $\alpha$ in $BH$, to $[\gamma^{-1}g(\alpha)\gamma]\in \pi_1BG=G$.
\end{enumerate}

\begin{thm}\label{thm: HNN graph of groups}
The space
\[
	BG\cup BH\times [0,1]
\]
given by attaching the two ends of the cylinder $BH\times [0,1]$ to $BG$ by the mappings $f,g$ is a $K(\pi,1)$ with $\pi=N(H,G,\f,\psi)$.
\end{thm}

The space $BG\cup BH\times[0,1]$ is an example of a ``graph of groups'' which is show to be a $K(\pi,1)$ in \cite{Hatcher}.

\begin{rem}
The isomorphism
\[
	N(H,G,\f,\psi)\cong \pi_1(BG\cup BH\times[0,1])
\]
is the inclusion map on $G=\pi_1BG$ and sends the generator $t$ of $N(H,G,\f,\psi)$ to the homotopy class of the path $\gamma^{-1}\beta$ where $\beta$ is the path $\beta(t)=(\ast,t)\in (\ast\times [0,1])\subseteq BH\times[0,1]$.
\end{rem}

We will fill in the details of this outline and show (Theorem \ref{thm: pi-1 G(S) is G(S)}) that $G(\cS)$ is the corresponding HNN extension of $G(\cS_0)$. We conclude that $B\cG(\cS)=K(G(\cS),1)$.

\subsection{Definitions and proofs}\label{ss 3.3: definitions and proofs}

Suppose that $\cS=\{\alpha\}$. Then $\cG(\cS)$ has two objects: $\cA(\emptyset)$ and $\cA(\alpha)$ and it has two nonidentity morphisms: $[\alpha]$ and $[-\alpha]:\cA(\alpha)\to \cA(\emptyset)$. Thus the classifying space is two points connected by two edges. This is a circle with fundamental group $\ZZ$. This is isomorphic to the group $G(\alpha)=\brk{x(\alpha)}$. So, $B\cG(\{\alpha\})=S^1=K(\ZZ,1)$.

The proof is by induction on $|\cS|$. Recall that $\cS$ is a finite convex set of real Schur roots. Then, the terms in the commutation relation for $x(\alpha),x(\beta)$ are in the set. So, the group $G(\cS)$ is defined.

\begin{lem}\label{lem: construction of S-0}
In any finite, nonempty, convex set of real Schur roots $\cS$ there is an $\omega\in\cS$ so that $\cS_0:=\cS\backslash \omega$ has the following properties.
\begin{enumerate}
\item $hom(\omega,\alpha)=0$ for all $\alpha\in\cS_0$.
\item $ext(\alpha,\omega)=0$ for all $\alpha\in\cS_0$.
\item $\cS_0$ is convex.
\end{enumerate}
\end{lem}

\begin{rem}\label{rem:unique map M to Momega m}
This implies that, for any $M\in\cA(\alpha_\ast)\in \cG(\cS)$, there is a uniquely determined exact sequence $M_0\cof M\onto M_\omega^m$ where $M_0\in\cA(\alpha_\ast\backslash\omega)\in\cG(\cS_0)$. Equivalently, whenever $\omega$ is an element of $\alpha_\ast$, it is a source in the quiver of $\cA(\alpha_\ast)$. So, any projective object in $\cA(\alpha_\ast\backslash\omega)$ is also projective in $\cA(\alpha_\ast)$. 
\end{rem}

\begin{proof}
Take a partial ordering of $\cS$ as given in the definition of convexity and let $\omega$ be any maximal element. Then (1), (2), (3) are clearly satified. 
\end{proof}

Since $\cS_0$ has one fewer element than $\cS$, the theorem is true for $\cS_0$. In other words, $B\cG(\cS_0)$ is $K(G(\cS_0),1)$. We will show that $G(\cS)$ is an HNN extension of $G(\cS_0)$ and that $B\cG(\cS)$ is a graph of groups for this group extension and therefore a $K(\pi,1)$ with $\pi=G(\cS)$.

Let $\cS_{\omega}$ be the set of all $\gamma\in\cS$  so that $hom(\gamma,\omega)=0$. In particular, $\gamma\neq\omega$. Since $ext(\gamma,\omega)=0$ for all $\gamma\in\cS$, this is a linear condition:
\begin{equation}\label{eq: def of S-omega}
	\cS_\omega=\{\gamma\in\cS\,|\, \brk{\gamma,\omega}=0\}
\end{equation}

\begin{lem}
Suppose that $\alpha,\beta\in \cS$ are hom perpendicular and $ext(\alpha,\beta)=0$.
\begin{enumerate}
\item If $\alpha,\beta\in\cS_{\omega}$ then $ab(\alpha,\beta)\subseteq \cS_{\omega}$. So, $\cS_{\omega}$ is convex.
\item If $\cS_\omega$ does not contain both $\alpha$ and $\beta$ then $\{\alpha,\beta\}\cap \cS_\omega=ab(\alpha,\beta)\cap \cS_\omega$.
\end{enumerate}
\end{lem}

\begin{proof} Since every element of $ab(\alpha,\beta)$ is a nonnegative linear combination of $\alpha,\beta$ the linear condition $\brk{-,\omega}=0$ holds on all elements if it holds for either $\alpha$ or $\beta$ and at least one other element. This proves (1) and (2) in the case when $\{\alpha,\beta\}\cap\cS_\omega$ is nonempty.

If $\alpha,\beta\notin \cS_\omega$ then $\brk{\alpha,\omega}>0$ and $\brk{\beta,\omega}>0$ so $\brk{\gamma,\omega}>0$ and thus $\gamma\notin\cS_\omega$ for any positive linear combination $\gamma$ of $\alpha,\beta$. This proves the remaining case of (2).
\end{proof}

\begin{prop}\label{prop: G-omega to GS is split mono}
The group homomorphism \[
G(\cS_{\omega})\into G(\cS)\]
induced by the inclusion $\cS_{\omega}\subseteq \cS$ has a left inverse. Since $\cS_{\omega}\subseteq \cS_0\subseteq \cS$, this implies that $G(\cS_\omega)$ is a retract of both $G(\cS_0)$ and $G(\cS)$.
\end{prop}

The homomorphism $\pi_1(\f_1)$ in the outline is the map $G(\cS_\omega)\into G(\cS_0)$ included by the inclusion $\cS_\omega\subseteq \cS_0$.

\begin{proof}
A retraction $r:G(\cS)\to G(\cS_\omega)$ can be defined as follows.
\[
	r(x(\alpha))=\begin{cases} x(\alpha) & \text{if } \alpha\in \cS_\omega\\
   1 & \text{otherwise}
    \end{cases}
\]
Since the relations in both groups are of the form $x(\alpha)x(\beta)=\prod x(\gamma_i)$ where the product is over all $\gamma_i\in ab(\alpha,\beta)$, the lemma shows that $r$ preserves relations.
\end{proof}

Let $\cH(\cS,\omega)$ be the full subcategory of $\cG(\cS)$ of all objects which do not lie in $\cG(\cS_0)$. These are $\cA=\cA(\beta_\ast)$ so that $\omega\in\beta_\ast$, i.e., $M_{\omega}$ is a simple object of $\cA$. The disjoint subcategories $\cH(\cS,\omega)$ and $\cG(\cS_0)$ of $\cG(\cS)$ together contain all the objects of $\cG(\cS)$. There are no morphisms from $\cG(\cS_0)$ to $\cH(\cS,\omega)$ and there are two types of morphisms from $\cH(\cS,\omega)$ to $\cG(\cS_0)$.

\begin{defn}\label{def: positive and negative morphisms}
By a \emph{negative morphism} we mean a cluster morphism $[T]:\cA(\alpha_\ast,\omega)\to\cA(\beta_\ast)$ from an object of $\cH(\cS,\omega)$ to an object of $\cG(\cS_0)$ so that $T$ contains the shifted projective object $P_\omega[1]$. A \emph{positive morphism} is a morphism $[T]:\cA(\alpha_\ast,\omega)\to\cA(\beta_\ast)$ with $\cA(\beta_\ast)\in\cG(\cS_0)$ which is not negative.
\end{defn}

We note that the target $\cA(\beta_\ast)$ of a negative morphism necessarily lies in $\cG(\cS_\omega)$. And any positive morphism $[T]$ must contain a module $T_0$ which maps onto $M_\omega$ since, otherwise, $|T|^\perp=\cA(\beta_\ast)$ would contain $M_\omega$.

\begin{prop}
The composition of any positive (resp. negative) morphism with any morphism in $\cG(\cS)$ is positive (resp. negative).
\end{prop}

We say that the positive morphisms form a \emph{two-sided ideal} in the category $\cG(\cS)$. The negative morphisms also form an ideal which is disjoint from the ideal of positive morphisms.

\begin{proof}
Suppose that $[T]:\cA(\alpha_\ast,\omega)\to\cA(\beta_\ast)$ is positive. Equivalently, $T$ contains some $T_0$ which maps onto the module $M_\omega$. Then any composition $[R]\circ [T]=[T,\sigma_T^{-1}R]$ will also contain $T_0$ and thus be positive. Also, any composition
\[
	[T]\circ[S]=[S,\sigma_S^{-1}T]:\cA(\alpha_\ast,\omega)\xrarrow{[S]} \cA(\beta_\ast,\omega)\xrarrow{[T]}\cA(\gamma_\ast)
\] will contain $\sigma_S^{-1}T_0\in \RR\alpha_\ast\oplus\RR\omega$ which is congruent to $T_0$ module $\RR S\subseteq\RR\alpha_\ast$ and therefore will have positive $\RR\omega$-coordinate. So, $[T]\circ[S]$ will be positive. The negative case is similar.
\end{proof}

\begin{lem}\label{lem: unique factorization of negative morphisms}
Any negative morphism $\cA(\alpha_\ast,\omega)\to \cA(\beta_\ast)$ factors uniquely through $[P_\omega[1]]:\cA(\alpha_\ast,\omega)\to\cA(\alpha_\ast)$.
\[
\xymatrixrowsep{15pt}\xymatrixcolsep{10pt}
\xymatrix{
\cA(\alpha_\ast,\omega) \ar[rr]\ar[dr]_{[P_\omega[1]]}& &\cA(\beta_\ast)\\
&\cA(\alpha_\ast)\ar@{-->}[ur]_{\exists![T]}
}
\]
\end{lem}

\begin{proof}
Any negative morphism has the form $[P_\omega,T]$ by definition. To be ext-orthogonal to $P_\omega[1]$ each $T_i\in T$ must lie in $\cA(\alpha_\ast)$. So, $\sigma_{P_\omega[1]}(T)=T$ is the unique partial cluster tilting set in $\cA(\alpha_\ast)$ so that $[T]\circ[P_\omega[1]]=[P_\omega[1],T]$.
\end{proof}

\begin{prop}\label{prop: isomorphism H=G(S-w)}
There is an isomorphism of categories 
\[
\f:\cH(\cS,\omega)\xrarrow\cong \cG(\cS_{\omega})
\]
given on objects by $\f\cA(\alpha_\ast,\omega)=\cA(\alpha_\ast)$ and on morphisms by $\f[T]=[T]$. Furthermore, inside the larger category $\cG(\cS)$, there is a natural transformation from the inclusion functor $\iota:\cH(\cS,\omega)\into \cG(\cS)$ to $\f:\cH(\cS,\omega)\cong \cG(\cS_\omega)\into \cG(\cS)$ given by $[P_\omega[1]]:\cA(\alpha_\ast,\omega)\to \cA(\alpha_\ast)$.
\end{prop}

\begin{proof} First, $\f$ is a bijection on objects since $\cA(\alpha_\ast)$ is an object of $\cG(\cS_\omega)$ if and only if each $\alpha_i$ is hom-orthogonal to $\omega$ which is equivalent to $\cA(\alpha_\ast,\omega)$ being in $\cH(\cS,\omega)$.

Let $P_\omega,P_{\alpha_i}$ be the relatively projective objects of $\cA(\alpha_\ast,\omega)$. Then each $P_{\alpha_i}\in\cA(\alpha_\ast)$. So, each shifted  projective object $P_{\alpha_i}[1]$ in $\cC(\alpha_\ast)$ lies in $\cC(\alpha_\ast,\omega)$. Thus, $\cC(\alpha_\ast)\subseteq\cC(\alpha_\ast,\omega)$.

A morphism $[T]:\cA(\alpha_\ast)\to\cA(\beta_\ast)$ in $\cG(\cS_\omega)$ is given by a partial cluster tilting set $T\subseteq \cC(\alpha_\ast)$ so that $|T|^\perp\cap\cA(\alpha_\ast)=\cA(\beta_\ast)$. \vs2

\noi\underline{Claim:} $|T|^\perp\cap\cA(\alpha_\ast,\omega)=\cA(\beta_\ast,\omega)$. So, $T$, considered as a partial cluster tilting set in $\cC(\alpha_\ast,\omega)$, gives a morphism $[T]:\cA(\alpha_\ast,\omega)\to\cA(\beta_\ast,\omega)$.\vs2

\noi Proof: Since $M_\omega\in T^\perp$ by definition of $\cS_\omega$, we have $|T|^\perp\cap\cA(\alpha_\ast,\omega)\supseteq\cA(\beta_\ast,\omega)$. Conversely, let $M\in |T|^\perp\cap\cA(\alpha_\ast,\omega)$. Then there is a short exact sequence $M_0\cof M\onto M_\omega^m$ where $M_0\in \cA(\alpha_\ast)$. Since $M,M_\omega\in |T|^\perp$, we must have $M_0\in |T|^\perp\cap \cA(\alpha_\ast)=\cA(\beta_\ast)$. But this implies that $M$ lies in $\cA(\beta_\ast,\omega)$ as required, proving the claim.
\vs2

Conversely, given any morphism $[T]:\cA(\alpha_\ast,\omega)\to \cA(\beta_\ast,\omega)$, we can compose with $[\overline P_\omega[1]]:\cA(\beta_\ast,\omega)\to\cA(\beta_\ast)$, where $\overline P_\omega$ is the projective cover of $M_\omega$ in $\cA(\beta_\ast,\omega)$, to get a negative morphism $\cA(\alpha_\ast,\omega)\to\cA(\beta_\ast)$. By the lemma, we get an induced morphism $\f[T]:\cA(\alpha_\ast)\to\cA(\beta_\ast)$ which is the unique morphism making the following diagram commute.
\[
\xymatrix{%begin xy matrix
\cA(\alpha_\ast,\omega)\ar[d]_{P_\omega[1]}\ar[r]^{[T]} &
	A(\beta_\ast,\omega)\ar[d]^{\overline P_\omega[1]}\\
\cA(\alpha_\ast)\ar[r]^{\f[T]=[T]} &
	A(\beta_\ast)
	}%end xy matrix
\]
This diagram implies at the same time that $\f$ is a functor and that $[P_\omega[1]]$ is a natural transformation. For example, given any morphism $[S]:\cA(\beta_\ast,\omega)\to\cA(\beta'_\ast,\omega)$ we have:
\[
	[\overline P'_\omega[1]]\circ [S]\circ [T]=\f [S]\circ [\overline P_\omega[1]]\circ [T]=\f[S]\circ\f[T]\circ[P_\omega[1]]
\]
showing that $\f([S]\circ[T])=\f[S]\circ\f[T]$. By the Claim proved above, $\f\cH(\cS,\omega)\to\cG(\cS_\omega)$ is an isomorphism of categories.
\end{proof}

We can now make precise the structure of the category $\cG(\cS)$ as given in the outline. The union of disjoint subcategories $\cH(\cS,\omega)\coprod \cG(\cS_0)$ contains all of the objects of $\cG(\cS)$ by definition. There are no morphisms from $\cG(\cS_0)$ to $\cH(\cS,\omega)$. The morphisms from $\cH(\cS,\omega)$ to $\cG(\cS_0)$ fall into two classes: positive and negative morphisms as defined above. Thus we have:

\begin{lem}\label{lem: G(S) is G+ cup G-} $\cG(\cS)$ is the union of two subcategories:
\[
	\cG(\cS)=\cG_+(\cS,\omega)\cup \cG_-(\cS,\omega)
\]
where $\cG_+(\cS,\omega)$ is the union of $\cH(\cS,\omega)\coprod \cG(\cS_0)$ with all positive morphisms and $\cG_-(\cS,\omega)$ is the union of $\cH(\cS,\omega)\coprod \cG(\cS_\omega)$ and all negative morphisms. (In $\cG_-$ we include only the targets of the negative morphisms.) So,
\[
	\cG_+(\cS,\omega)\cap \cG_-(\cS,\omega)=\cH(\cS,\omega)\smallcoprod \cG(\cS_\omega).
\]
\end{lem}

From the definition of the classifying space of a category, we will obtain:
\begin{lem}\label{lem: decomposition of BG} We have an analogous decomposition of the topological space $B\cG(\cS)$:
\[
	B\cG(\cS)=B\cG_+(\cS,\omega)\cup B\cG_-(\cS,\omega)
\]
\[
	B\cG_+(\cS,\omega)\cap B\cG_-(\cS,\omega)=B\cH(\cS,\omega)\smallcoprod B\cG(\cS_\omega).
\]
\end{lem}

By the unique factorization of negative morphisms given in Lemma \ref {lem: unique factorization of negative morphisms}, we then show:

\begin{lem}\label{lem: BG- is a cylinder}
The classifying space $B\cG_-(\cS,\omega)$ is homeomorphic to a cylinder:
\[
	B\cG_-(\cS,\omega)=B\cH(\cS,\omega)\times[0,1]
\]
The end $B\cH(\cS,\omega)\times0$ of this cylinder is attached to $B\cG_+(\cS,\omega)$ by the inclusion $B\cH(\cS,\omega)\subseteq B\cG_+(\cS,\omega)$ and the other end by the mapping $B\f:B\cH(\cS,\omega)\cong B\cG(\cS_\omega)\subseteq B\cG_+(\cS,\omega)$ induced by the functor $\f:\cH(\cS,\omega)\cong \cG(\cS_\omega)$.
\end{lem}

In another key lemma proved below, we will see that $B\cG_+(\cS,\omega)$ is homotopy equivalent to $B\cG(\cS_0)$ which is a $K(\pi,1)$ by induction since $|\cS_0|=|\cS|-1$. We will then use Theorem \ref{thm: HNN graph of groups} to conclude that $B\cG(\cS)$ is a $K(\pi,1)$.

\subsection{Classifying space of a category and Lemmas \ref {lem: decomposition of BG}, \ref{lem: BG- is a cylinder}}\label{ss 3.4: classifying space of a category}

We first recall the definition of the classifying space of a category.

\subsubsection{Classifying space of a category}

The classifying space of any small category $\cC$ is the geometric realization of its \emph{nerve}: $B\cC=|\cN_\bullet\cC|$ where $\cN_\bullet\cC$ is the simplicial set which in degree $n$ is the set of all sequences of $n$ composable morphisms in $\cC$:
\[
	\cN_n\cC:=\coprod_{X,Y\in\cC}\cC_n(X,Y)
\]
where $\cC_n(X,Y)$ is the set of all directed paths of length $n$ from $X$ to $Y$ in $\cC$:
\[
	\cC_n(X,Y):=\left\{
	X=X_0\xrarrow{f_1} X_1\xrarrow{f_2} X_2\xrarrow{f_3} \cdots\xrarrow{f_n} X_n=Y
	\right\}
\]
To simplify notation and clarify the case $n=0$, we will sometimes add redundant information to the elements of the set $\cC_n(X,Y)$. Namely, we add all compositions of morphisms and all identity morphisms of all objects $X_i$ in the sequence. Then, a {path of length $n$} in $\cC$ becomes a collection of morphisms $f_{ij}:X_i\to X_j$ for $0\le i\le j\le n$ so that $f_{jk}\circ f_{ij}=f_{ik}$ for all $0\le i\le j\le k\le n$, and so that $f_{ii}$ is the identity morphism of $X_i$ for each $i$. When $n=0$ we have only the identity morphism $f_{00}$ of $X=X_0=Y$. (So, $\cC_0(X,Y)$ is empty when $X\neq Y$.)

The simplicial structure maps for $\cN_\bullet\cC$ are given as follows. Let $[n]:=\{0,1,\cdots,n\}$. Then for any set mappings $a:[n]\to [m]$ so that $0\le a(i)\le a(j)\le m$ for all $0\le i\le j\le n$ we have the mapping $a^\ast:\cN_m\cC\to\cN_n\cC$ given by
\[
	a^\ast((p,q)\mapsto f_{pq})=((i,j)\mapsto f_{a(i)a(j)})
\]

The \emph{classifying space} of $\cC$ is the geometric realization of $\cN_\bullet\cC$ which is the topological space given by
\[
	B\cC=|\cN_\bullet\cC|:=\coprod_{n\ge0} \cN_n \cC\times \Delta^n/\sim
\]
with the quotient topology where $\Delta^n$ is the standard $n$-simplex with vertices $v_0,\cdots,v_n$ and the equivalence relation is given by
\[
	(f,a_\ast(t))\sim (a^\ast f,t)
\]
for all $f\in \cN_m\cC$, $t\in\Delta^n$ and $a:[n]\to[m]$. The mapping $a_\ast:\Delta^n\to\Delta^m$ is the unique affine linear mapping which sends $v_i$ to $v_{a(i)}$ for all $i\in[n]$.

\subsubsection{Proof of Lemma \ref{lem: decomposition of BG}} By definition,
\[
	B\cG(\cS)=\coprod_n \cN_n\cG(\cS)\times\Delta^n/\sim
\]
So, $\cN_n\cG(\cS)$ is the disjoint union of four sets: $\cN_n\cH(\cS,\omega)$, $\cN_n\cG(\cS_0)$, the set of all paths $(f_{ij})$ which included one positive morphism, call these \emph{negative paths}, and the set of all paths including one negative morphism, call these \emph{positive paths}.

But, all positive paths lie in $\cN_n\cG_+(\cS,\omega)$ and all negative paths lie in $\cN_n\cG_-(\cS,\omega)$. Also, $\cN_n\cG_+(\cS,\omega)$ contains $\cN_n\cH(\cS,\omega)\coprod\cN_n\cG(\cS_0)$. Therefore,
\[
	B\cG(\cS)=B\cG_+(\cS,\omega)\cup B\cG_-(\cS,\omega)
\]

Since a sequence of composable morphisms in $\cG(\cS)$ contains at most one morphism not in $\cH(\cS,\omega)$ or $\cG(\cS_0)$, a path cannot be both positive and negative. So any element of $\cN_n\cG_+(\cS,\omega)\cap \cN_n\cG_-(\cS,\omega)$ lies in $\cN_n\cH(\cS,\omega)$ or $\cN_n\cG(\cS_\omega)$. Therefore,
\[
	B\cG_+(\cS,\omega)\cap B\cG_-(\cS,\omega)=B\cH(\cS,\omega)\smallcoprod B\cG(\cS_\omega)
\]
completing the proof of Lemma \ref{lem: decomposition of BG}.

\subsubsection{Proof of Lemma \ref{lem: BG- is a cylinder}}

We will use the following well-know construction of the cylinder of a category. Let $\cI$ be the category with two objects $0,1$ and exactly one nonidentity morphism $d:0\to 1$. Then, it is easy to see that $B\cI$ is the unit interval $[0,1]$. Since $B(\cC\times\cD)=B\cC\times B\cD$ for any two small categories $\cC,\cD$, we get:
\[
	B(\cC\times \cI)=B\cC\times[0,1]
\]
with two ends given by $B(\cC\times 0)=B\cC\times 0$ and $B(\cC\times 1)=B\cC\times 1$. 

To prove Lemma \ref{lem: BG- is a cylinder} it therefore suffices to construct an isomorphism of categories:
\[
	\Phi:\cH(\cS,\omega)\times \cI\cong \cG_-(\cS,\omega)
	\]
Such an isomorphism is given on objects by $\Phi(\cA,0)=\cA$ for all $\cA=\cA(\alpha_\ast,\omega)\in\cH(\cS,\omega)$ and $\Phi(\cA,1)=\f \cA=\cA(\alpha_\ast)$. On morphisms, $\Phi$ is given by $\Phi([T],id_i)=[T]$ for $i=0,1$ and $\Phi([T],d)=[P_\omega[1],T]$. It is easy to see that $\Phi$ is a functor, that it is the inclusion functor on $\cH(\cS,\omega)\times 0$ and $\f$ on $\cH(\cS,\omega)\times 1$.

The inverse of $\Phi$ is $\Psi:\cG_-(\cS,\omega)\to \cH(\cS,\omega)\times\cI$ given as follows.
\begin{enumerate}
\item $\Psi \cA(\alpha_\ast,\omega)=(\cA(\alpha_\ast,\omega),0)$ for all $\cA(\alpha_\ast,\omega)\in\cH(\cS,\omega)$.
\item $\Psi \cA(\beta_\ast)=(\cA(\beta_\ast,\omega),1)$ for all $\cA(\beta_\ast)\in\cG(\cS_\omega)$.
\item $\Psi[T]=([T],id_0)$ for all $[T]:\cA(\alpha_\ast,\omega)\to \cA(\beta_\ast,\omega)$ in $\cH(\cS,\omega)$.
\item $\Psi[T]=(\f^{-1}[T],id_1)$ for all $[T]:\cA(\alpha_\ast)\to \cA(\beta_\ast)$ in $\cG(\cS_\omega)$ where $\f^{-1}[T]=[T]$ considered as a morphism $\cA(\alpha_\ast,\omega)\to \cA(\beta_\ast,\omega)$. (See Proposition \ref{prop: isomorphism H=G(S-w)}.)
\item $\Psi$ takes $[P_\omega[1],T]:\cA(\alpha_\ast,\omega)\to\cA(\beta_\ast)$ to $([T],d):(\cA(\alpha_\ast,\omega),0)\to(\cA(\beta_\ast,\omega),1)$.
\end{enumerate}
It follows from Lemma \ref{lem: unique factorization of negative morphisms} and Proposition \ref{prop: isomorphism H=G(S-w)} that $\Psi$ is well-defined and inverse to $\Phi$. This proves Lemma \ref{lem: BG- is a cylinder}.

\subsection{Key lemma}\label{ss 3.5: key lemma}
We will now prove the key lemma:

\begin{lem}\label{lem: key lemma}
The inclusion functor $j:\cG(\cS_0)\into \cG_+(\cS,\omega)$ induces a homotopy equivalence to $Bj: B\cG(\cS_0)\simeq B\cG_+(\cS,\omega)$.
\end{lem}

The proof uses Quillen's Theorem A which we now review.

Given any functor $\psi :\cC\to\cD$ between small categories $\cC,\cD$, the \emph{fiber category} $X\backslash \psi $ of $\psi $ over any object $X$ in $\cD$ is defined to be the category of all pairs $(Y,f)$ where $Y\in\cC$ and $f:X\to \psi Y$ is a morphism of $\cD$. A morphism $(Y,f)\to (Z,g)$ in $X\backslash \psi $ is defined to be a morphism $h:Y\to Z$ in $\cC$ so that $g=\psi h\circ f:X\to \psi Y\to \psi Z$.

\begin{thm}[Quillen's Theorem A]\cite{Quillen} If $B(X\backslash \psi )$ is contractible for every $X\in\cD$ then the mapping $B\psi :B\cC\to B\cD$ is a homotopy equivalence.
\end{thm}

\begin{rem} 
By a common abuse of language we will often say that a category is \emph{contractible} if its classifying space is contractible and a functor is a \emph{homotopy equivalence} if it induces a homotopy equivalence on classifying spaces.\end{rem}

%\begin{proof} (This proof will be commented out.) Embed the fiber category into the larger ``undercategory'' $\cD\backslash \psi $. This is the category of all triples $(X,Y,f)$ where $X\in\cD,Y\in\cC$ and $f:X\to \psi Y$ is a morphism in $\cD$. Morphisms $(X,Y,f)\to (X',Y',f')$ are pairs of morphisms $X\to X'$ and $Y\to Y'$ making the obvious diagram commute. \\ \indent The forgetful functor $F:\cD\backslash \psi \to \cC$ sending $(X,Y,f)$ to $Y$ has a right adjoint $G:\cC\to \cD\backslash \psi $ given by $GZ=(\psi Z,Z,id_{\psi Z})$. So, $F$ is a homotopy equivalence. So, it suffices to show that $\psi\circ F:\cD\backslash \psi \to \cD$ is a homotopy equivalence. But this functor is homotopic to the functor $\pi:\cD\backslash \psi \to \cD$ sending $(X,Y,f)$ to $X$ with homotopy given by $f:\pi(X,Y,f)=X\to \psi Y=\psi\circ F(X,Y,f)$. So, it suffices to show that $\pi$ is a homotopy equivalence.\\ \indent We claim that $B\pi$ has the property that the inverse image of every simplex in $B\cD$ is contractible in $B(\cD\backslash \psi)$. This is because the inverse image of the simplex $X_0\to X_1\to \cdots\to X_k$ is homotopy equivalent $X_k\backslash \psi$ which is assumed to be contractible. Therefore $B\pi$ is a homotopy equivalence by the topological acyclic carrier theorem. (Given a set of contractible cell complexes closed under nonempty intersection, their union is homotopy equivalent to the poset of these cell complexes ordered by inclusion.)\end{proof}

To prove the key lemma it therefore suffices to show that the fiber category $\cA_0\backslash j$ is contractible for every fixed object $\cA_0\in\cG_+(\cS,\omega)$. There are two cases. Either $\cA_0\in \cG(\cS_0)$ or $\cA_0\in\cH(\cS,\omega)$. In the first case, $\cA_0\backslash j$ is contractible since it has an initial object given by $(\cA_0,\cA_0,id_{\cA_0})$. Therefore, we assume $\cA_0=\cA(\alpha_\ast,\omega)\in\cH(\cS,\omega)$.

The fiber category $\cA_0\backslash j$ is the category of all positive morphisms $[T]:\cA_0\to\cB\in\cG(\cS_0)$. The elements of $\cN_k(\cA_0\backslash j)$ are equivalent to commuting diagrams:
\[
\xymatrix{
& \cA(\alpha_\ast,\omega)\ar[dl]\ar[d]\ar[drr]\\
\cB_0 \ar[r]& \cB_1\ar[r] &\cdots\ar[r]  &  \cB_k 
}
\]
where each $\cB_i\in\cG(\cS_0)$ and each arrow $\cA(\alpha_\ast,\omega)\to \cB_i$ is a positive morphism. Such diagrams are in bijection with filtrations $T_0\subseteq T_1\subseteq \cdots\subseteq T_k$ of nonempty partial cluster tilting sets in $\cC(\alpha_\ast,\omega)$ which have the following two properties.
\begin{enumerate}
\item $T_k$ does not contain $P_\omega[1]$.
\item $T_0$ contains a module which maps onto $M_\omega$. (Equivalently, $T_0\not\subseteq\cA(\alpha_\ast)$.) 
\end{enumerate}
Using this description we will show that the simplicial set $\cN_\bullet(\cA_0\backslash j)$ is isomorphic to a familiar simplicial complex.

Suppose that $\alpha_\ast=\{\alpha_1,\cdots,\alpha_n\}$ has $n$ elements. Then every cluster tilting set in the finite set $\cC(\alpha_\ast,\omega)$ has $n+1$ elements and every subset of every cluster tilting set is a partial cluster tilting set by definition. Therefore, the set of nonempty partial cluster tilting sets is an $n$-dimensional simplicial complex which we denote $K^n$. By \cite{IOTW3}, $|K^n|$ is homeomorphic to the $n$-sphere $S^n$.

\begin{lem} {Nonempty partial cluster tilting sets in $\cC(\alpha_\ast,\omega)$ which do not contain $P_\omega[1]$ form a subcomplex $E^n$ of $K^n$ whose realization is homeomorphic to a closed $n$-disk $D^n$.}
\end{lem}

\begin{proof}
$P_\omega[1]$ is a single vertex of $K^n$ and its link is given by all nonempty partial cluster tilting sets in $\cC(\alpha_\ast)$. This forms an $n-1$ sphere which divides $|K^n|=S^n$ into two halves. The half containing $P_\omega[1]$ is a cone on $S^{n-1}$ and thus standard. This implies that the other half, which is $|E^n|$ is also standard and thus an $n$-disk.
\end{proof}

Note that the boundary of $|E^n|=D^n$ is the link of $P_\omega[1]$ in $K^n$. We denote the corresponding subcomplex of $E^n$ by $\partial E^n$. Then $\partial E^n$ is the set of all nonempty cluster tilting sets in $\cC(\alpha_\ast)$.

Let $Simp(E^n)$ be the ``poset category'' whose objects are the simplices of $E^n$ with one morphism $\sigma\to\tau$ whenever $\sigma\subseteq\tau$. Recall that the \emph{first barycentric subdivision} of $E^n$ is $sdE^n=\cN_\bullet Simp(E^n)$.

\begin{lem}
The fiber category $\cA(\alpha_\ast,\omega)\backslash j$ is isomorphic to the full subcategory $J$ of $Simp(E^n)$ consisting of all simplices $\sigma$ which are not contained in $\partial E^n$.
\end{lem}

\begin{proof}
The objects of $\cA(\alpha_\ast,\omega)\backslash j$ are nonempty partial cluster tilting sets $[T]$ with two additional conditions listed earlier. If we ignore the conditions, we have a poset category isomorphic to $Simp(K^n)$ by definition. Adding the first condition give the full subcategory $Simp(E^n)$. Adding the second condition gives the full subcategory $J$.
\end{proof}

The key lemma now follows from the following elementary topological fact whose proof 
%will be commented out.
is left as an easy exercise.

\begin{prop} Let $E^n$ be a simplicial complex whose geometric realization $|E^n|$ is homeomorphic to the standard $n$-disk $D^n$. Let $J$ be the subcomplex of the first barycentric subdivision $sdE^n$ spanned by all barycenters $b_\sigma$ of simplices $\sigma$ of $E^n$ which are not contained in the boundary of $D^n$. Then $|J|$ is contractible.
\end{prop}

%\begin{proof}% (Commented out)Let $\lambda:|sdE^n|\to [0,1]$ be the piecewise linear function given on vertices by $\lambda(b_\sigma)=0$ if $\sigma\subset\partial D^n$ and $\lambda(b_\sigma)=1$ otherwise. Then $U:=\{x\in|sdE^n|\,:\, \lambda(x)>0$ is the interior of $D^n$ and is thus contractible. Any $x\in U$ is given in barycentric coordinates by $x=\sum s_iv_i+\sum t_jw_j$ where $v_i,w_j$ are the vertices of the smallest simplex in $sdE^n$ containing $x$ and $v_i$ are those vertices which lie in $\partial D^n$. Then $\lambda(x)=\sum t_j\neq0$. Then a retraction $r:U\to |J|$ of $U$ onto $|J|$ is given by $r(x)=\frac1{\lambda(x)}\sum t_jw_j$. Since any retract of a contractible space is contractible, we conclude that $|J|$ is contractible.\end{proof}

\subsection{$G(\cS)$ is an HNN extension of $G(\cS_0)$}\label{ss 3.6: G(S) is HNN ext of G(S0)}

We will show that
\[
	G(\cS)=N(G(\cS_\omega),G(\cS_0),\f,\psi)
\]
where $G=G(\cS_0)$, $H=G(\cS_\omega)$, $\f:G(\cS_\omega)\into G(\cS_0)$ is the monomorphism induced by the inclusion $\cS_\omega\subseteq \cS_0$ (see Proposition \ref{prop: G-omega to GS is split mono}) and  $\psi:G(\cS_\omega)\into G(\cS_0)$ is a monomorphism which we now construct. The key step is the following theorem where we use the shorthand notation $[\beta]:=[M_\beta]$ and $[-\beta]:=[M_\beta[1]]$.

\begin{thm}\label{thm: pi-1 G(S) is G(S)}
Taking the zero category $0=\cA(\emptyset)$ as basepoint for $B\cG(\cS)$, we have an isomorphism of groups $G(\cS)\cong \pi_1B\cG(\cS)$ given by sending each generator $x(\beta),\beta\in\cS$ of $G(\cS)$ to (the homotopy class of) the loop in $B\cG(\cS)$ at $\cA(\emptyset)$ given by
\[
	\cA(\emptyset)\xleftarrow{[-\beta]} \cA(\beta)\xrarrow{[\beta]}\cA(\emptyset)
\]
(going from left to right).\end{thm}

If $\cS=\{\beta\}$ then this is true since the loop is the entire category. So, we can assume this holds for $\cS_0$ and $\cS_\omega$ by induction on the size of $\cS$. Since $\cH(\cS,\omega)\cong \cG(\cS_\omega)$, we get the following corollary which we will use to prove the proposition.

\begin{cor}\label{cor: paths with pasepoint omega}
Taking $\cA(\omega)$ as basepoint for $B\cH(\cS,\omega)$, we have an isomorphism of groups $G(\cS_\omega)\cong \pi_1B\cH(\cS,\omega)$ given by sending each generator $x(\alpha), \alpha\in\cS_\omega$ of $G(\cS_\omega)$ to (the homotopy class of) the loop in $B\cH(\cS,\omega)$ at $\cA(\omega)$ given by
\[
	\cA(\omega)\xleftarrow{[-\alpha]} \cA(\alpha, \omega)\xrarrow{[\alpha]}\cA(\omega)
\]
(going from left to right). 
\end{cor}

We define $\psi:G(\cS_\omega)\to G(\cS_0)$ to be the homomorphism:
\[
	G(\cS_\omega)\cong \pi_1 B\cH(\cS,\omega) \to \pi_1B\cG_+(\cS,\omega)\cong\pi_1 B\cG(\cS_0)=G(\cS_0) 
\]
induced by the inclusion functors $\cH(\cS,\omega)\into \cG_+(\cS,\omega)$ and $\cG(\cS_0)\into \cG_+(\cS,\omega)$ and by the choice of paths $\gamma=[\omega]$ as explained below. 

First, we recall that, when a continuous mapping $f:X\to Y$ fails to take the basepoint $x_0\in X$ to the basepoint $y_0\in Y$, we need to choose a path $\gamma$ from $f(x_0)$ to $y_0$ in order to get an induced map on fundamental groups. Then, for any $[\alpha]\in\pi_1(X,x_0)$, we define $\pi_1(f,\gamma)[\alpha]\in\pi_1(Y,y_0)$ to be the homotopy class of the loop at $y_0$ given by $\gamma^{-1}f(\alpha)\gamma$.

In our case we take $\gamma$ to be the path in $B\cG_+(\cS,\omega)$ from the base point $\cA(\omega)$ of $B\cH(\cS,\omega)$ to the basepoint $\cA(\emptyset)$ of $B\cG(\cS_0)$ given by the positive morphism $[\omega]:\cA(\omega)\to\cA(\emptyset)$.

\begin{prop}\label{prop: psi has left inverse}
The homomorphism $\psi:G(\cS_\omega)\to G(\cS_0)$ has a left inverse and is therefore a monomorphism. Furthermore, $\psi$ is given on generators $x(\alpha)$ for $\alpha\in\cS_\omega$ by
\begin{equation}\label{eq: equation for psi(x(a))}
	\psi(x(\alpha))=\prod x(\gamma_i)
\end{equation}
where $\gamma_i$ runs over all real Schur roots of the form $\gamma_i=a_i\alpha+b_i\omega$ where $a_i>0$ and the product is taken in decreasing order of the ratio $b_i/a_i$.
\end{prop}

\begin{proof} We show that the second statement implies the first.
Let $\iota:G(\cS_0)\to G(\cS)$ be the homomorphism induced by the inclusion $\cS_0\into \cS$. Let $\phi$ be the automorphism of $G(\cS)$ given by conjugation by $x(\omega)$. Thus $\phi(g)=x(\omega)gx(\omega)^{-1}$. Then, by the defining relations of $G(\cS)$, we have
\[
	\phi\circ \iota\circ\psi(x(\alpha))=x(\omega)\left(\smallprod x(\gamma_i)\right)x(\omega)^{-1}=x(\alpha)
\]
Therefore $\phi\circ\iota\circ\psi:G(\cS_\omega)\to G(\cS)$ is the split monomorphism with left inverse $r$ and $r\circ\phi\circ \iota$ is a left inverse for $\psi$.

It remains to prove the equation \eqref{eq: equation for psi(x(a))}. Since $\psi$ is defined in terms of the inclusion functor $\cH(\cS,\omega)\into \cG_+(\cS,\omega)$, we need to look at the positive morphisms $\cA(\alpha,\omega)\to \cA(\emptyset)$. These are given by all cluster tilting sets in $\cC(\alpha,\omega)$ which do not include $P_\omega[1]$. Since $\cC(\alpha,\omega)$ is finite, there are six possible cases: $A_1\times A_1,A_2,B_2,B_2^{op}=C_2,G_2,G_2^{op}$. We will use type $C_2$ as an example. The other cases are very similar. 

When we say that $\cC(\alpha,\omega)$ has type $C_2$ we mean that the division ring $F_\alpha$ is a degree two extension of $F_\omega$. The Auslander-Reiten quiver of the category $\cA(\alpha,\omega)$ has four objects:
\[\xymatrixrowsep{10pt}\xymatrixcolsep{10pt}
\xymatrix{
& P_\omega \ar[dr]& &  I_\omega\\
P_\alpha \ar[ur] &  & I_\alpha\ar[ur] 
}
\]
with dimension vectors $\alpha,\beta,\gamma,\omega$, respectively, where $\beta=\alpha+\omega$ and $\gamma=\beta+2\omega$. The objects of $\cC(\alpha,\omega)$ are $P_\omega,P_\alpha,I_\alpha,I_\omega,P_\alpha[1],P_\omega[1]$. Of these, the first five give all positive morphisms from $\cA(\alpha,\omega)$ to a wide category of rank 1. Consecutive pairs from these first five objects give all four positive morphisms $\cA(\alpha,\omega)\to \cA(\emptyset)$, each of which can be factored in two ways. This gives the following commuting diagram in $\cG_+(\cS,\omega)$.
\[
\xymatrix{%begin xy matrix
\cA(\omega) \ar[d]_{[\omega]}
&& \cA(\alpha,\omega) \ar[ll]_{[-\alpha]}\ar[rr]^{[\alpha]}
\ar[d]^{[\gamma]}
\ar[dl]_{[\omega]}
\ar[dr]^{[\beta]}
&& \cA(\omega)\ar[d]^{[\omega]}
 \\
\cA(\emptyset) 
& \cA(\gamma) \ar[d]_{[\gamma]}\ar[l]_{[-\gamma]}
& \cA(\beta) \ar[dr]_{[\beta]}\ar[dl]^{[-\beta]}
& \cA(\alpha) \ar[d]^{[-\alpha]}\ar[r]^{[\alpha]}
& \cA(\emptyset)\\
& \cA(\emptyset) 
&& \cA(\emptyset)
	}%end xy matrix
\]
The homomorphism $\psi$ sends $x(\alpha)$ first to the loop at $\cA(\omega)$ given by the top row of the diagrams as in Corollary \ref{cor: paths with pasepoint omega}:
\[
	\cA(\omega)\xleftarrow{[-\alpha]}\cA(\alpha,\omega)\xrarrow{[\alpha]}\cA(\alpha)
\]
then to the loop at $\cA(\emptyset)$ given by the path 
\[
\cA(\emptyset)\xrarrow{[\omega]^{-1}}\cdot\xrarrow{[-\alpha]^{-1}} \cdot\xrarrow{[\alpha]} \cdot\xrarrow{[\omega]}\cA(\emptyset)
\]
which is homotopic to the path $[-\gamma]^{-1}[\gamma][-\beta]^{-1}[\beta][-\alpha]^{-1}[\alpha]$. In other words,
\[
	\psi(x(\alpha))=x(\gamma)x(\beta)x(\alpha)
\]
These correspond to the objects in the AR quiver of $\cA(\alpha,\omega)$ in reverse order starting from the (relatively) injective module $I_\alpha$ and ending in the (relatively) simple projective module $P_\alpha$ in all cases. Therefore, \eqref{eq: equation for psi(x(a))} holds in all cases. Our proposition follows.
\end{proof}

Recall that we are assuming by induction that Theorem \ref{thm: pi-1 G(S) is G(S)} holds for $\cS_0$ and $\cS_\omega$ by induction on $|\cS|$.

\begin{cor}\label{cor: G(S) is HNN extension}
Let $\cS=\cS_0\cup\{\omega\}$ be as above. Then $G(\cS)$ is isomorphic to the HNN extension $N(G(\cS_\omega),G(\cS_0),\iota,\psi)$ where $\iota:G(\cS_\omega)\into G(\cS_0)$ is the inclusion map and $\psi:G(\cS_\omega)\into G(\cS_0)$ is the split monomorphism described above. The isomorphism 
\[
	N(G(\cS_\omega),G(\cS_0),\iota,\psi)\cong G(\cS)
\]
is the inclusion map on $G(\cS_0),G(\cS_\omega)$ and sends the new generator $t$ to $x(\omega)^{-1}$.
\end{cor}

\begin{proof}
The HNN extension $N(G(\cS_\omega),G(\cS_0),\iota,\psi)$ adds one generator $t^{-1}=x(\omega)$ to $G(\cS_0)$ and, for each $\alpha\in\cS_0$, the new relation
\[
	x(\alpha)=x(\omega)\psi(x(\alpha))x(\omega)^{-1}
\]
By \eqref{eq: equation for psi(x(a))}, this is equivalent to the relation
\[
	x(\alpha)x(\omega)=\prod x(\gamma_i)
\]where $\gamma_i$ runs over all real Schur roots of the form $\gamma_i=a_i\alpha+b_i\omega$ including the case $a_i=0$ and the product is taken in decreasing order of the ratio $b_i/a_i$. These are the defining relations of $G(\cS)$ which are not in $G(\cS_0)$, proving the corollary.
\end{proof}

\begin{proof}[Proof of Theorem \ref{thm: pi-1 G(S) is G(S)}]
We have completed the proofs of all statement in the outline in Section \ref{ss 3.2: outline of G(S)}. Therefore, by Theorem \ref{thm: HNN graph of groups}, $B\cG(\cS)$ is a $K(\pi,1)$ with $\pi$ equal to the HNN extension $N(G(\cS_\omega),G(\cS_0),\iota,\psi)$ which is equal to $G(\cS)$ with generators $x(\alpha)\in G(\cS)$ corresponding to either $x(\alpha)\in G(\cS_0)$ or to $t^{-1}$ by Corollary \ref{cor: G(S) is HNN extension} above. This proves the theorem for all finite convex $\cS$.
\end{proof}

The proof above also completes the proof of the main Theorem \ref{thm 3.5: G(S) is K(pi,1)}.

\section{Picture groups}\label{sec 4: Picture groups}

We will show that, when $\Lambda$ has finite representation type, the classifying space of the cluster morphism category of $mod$-$\Lambda$ is the CW-complex associated to the algebra in \cite{IOTW4} using pictures. This cell complex has one $k$-cell $e(\cA)$ for every wide subcategory of $mod\text-\Lambda$ of rank $k$. We extend this construction to a space $X(\cS)$ for every finite convex set $\cS$ of real Schur roots and show that $X(\cS)$ is homeomorphic to $B\cG(\cS)$. We will write $\vare (\cA)$ for the cell in $B\cG(\cS)$ corresponding to $e(\cA)\subseteq X(\cS)$. We will also construct the cellular chain complex of $X(\cS)\simeq B\cG(\cS)$ to be used in later papers.

\subsection{Construction of the CW-complex $X(\cS)$}\label{ss 4.1: the CW-complex X(S)}

For every object $\cA$ in $\cG(\cS)$ we will construct a simplicial complex whose geometric realization $E(\cA)$ is homeomorphic to a disk of dimension equal to the rank of $\cA$. There is a continuous mapping $E(\cA)\to B\cG(\cS)$ which is an embedding on the interior of $E(\cA)$ and $B\cG(\cS)$ will be the disjoint union of the images $\vare(\cA)$ of these interiors. When $\cC(\cA)$ is not finite, $E(\cA)$ is not compact and therefore cannot be homeomorphic to a disk and our construction would not give a CW-complex. Therefore, finiteness of $\cS$ is essential for this construction.

Suppose $\cA=\cA(\alpha_\ast)$ with rank $n$. Then the set of real Schur roots in $\ZZ\alpha_\ast\cong \ZZ^n$, being finite by the assumption that they all lie in the finite set $\cS$, is the root system $\Phi(\alpha_\ast)$ of a disjoint union of Dynkin quivers which form the valued quiver associated to $\cA$. Let $K(\cA)$ be the simplicial complex whose vertices are the positive roots $\Phi_+(\alpha_\ast)$ and the negative projective roots in $\Phi(\alpha_\ast)$. These are the dimension vectors of the objects of $\cC(\cA)$. A set of vertices span a simplex in $K(\cA)$ if they are pairwise ext-orthogonal. It is well-known (see \cite{IOTW3}) that the geometric realization $|K(\cA)|$ is homeomorphic to the $n-1$ sphere. For example, when $n=1$, there are only two roots $\alpha,-\alpha$ and $|K(\cA)|=S^0$ is two points.

Let $\simp_+K(\cA)$ be the poset category of simplices in $K(\cA)$ ordered by inclusion, including the empty simplex. Let $\simp K(\cA)$ be the full subcategory of nonempty simplices. The classifying space $B\simp K(\cA)$ is the first barycentric subdivision of $K(\cA)$ and $B\simp_+ K(\cA)$, being the cone on $B\simp K(\cA)$ is a triangulated $n$ disk. We define
\[
	E(\cA):=B\simp_+K(\cA)\cong D^n.
\]

We define the \emph{picture space} $X(\cS)$ to be the union of cells:
\[
	X(\cS)=\coprod_{\cA\in\cG(\cS)}E(\cA)/\sim
\]
with identifications given as follows.

For every cluster morphism $[T]:\cA\to\cB$ in the category $\cG(\cS)$ of rank $rk\,\cA-rk\,\cB=k$ we have the embedding
\[
	\sigma_T:\cC(\cB)\into \cC(\cA)
\]
with image $\cC_T(\cA)$ so that $X,Y\in\cC(\cB)$ are ext-orthogonal if and only if $\sigma_TX,\sigma_TY$ are ext-orthogonal in $\cC(\cA)$. This induces an embedding of categories:
\[
	\Sigma_T:\simp_+K(\cB)\to \simp_+K(\cA)
\]
which sends every $p$-simplex $X$ in $\simp_+K(\cB)$ ($p\ge-1$) to the $(p+k)$-simplex \[
	\Sigma_TX=\sigma_TX\cup T
\]
in $\simp_+K(\cA)$. In particular, it sends the cone point in $\simp_+K(\cA)$ to the $k-1$ simplex spanned by the $k$ objects of $T$.

\begin{lem}\label{lem: Sigma is a functor}
Given $[T]:\cA\to \cB$ and $[S]:\cB\to\cC$ with composition $[S]\circ[T]=[T\cup \sigma_TS]:\cA\to\cC$, we have
\[
	\Sigma_{T\cup\sigma_TS}=\Sigma_T\Sigma_S
\]
\end{lem}
\begin{proof} For any $X$ in $\simp_+\cC$ we have
\[
\Sigma_T\Sigma_SX=\Sigma_T(S\cup \sigma_SX)=T\cup \sigma_TS\cup \sigma_T\sigma_SX=\sigma_{T\cup \sigma_TS}X
\]
since $\sigma_T\sigma_S=\sigma_{T\cup \sigma_TS}$ \eqref{eq: sigma TS=sigma T sigma S}.
\end{proof}

On classifying spaces, this gives an embedding of cells:
\[
	B\Sigma_T:E(\cB)=B\simp_+K(\cB)\to B\simp_+K(\cA)=E(\cA)
\]
which sends the center of $E(\cB)$ to the barycenter of the $k-1$ simplex spanned by $T$.

Let $\overline e(\cA)$, $e(\cA)$ be the images of $E(\cA)$ and its interior in $X(\cS)$. Then the statement that the quotient space
\[
	\bigcup \overline e(\cA)=\coprod_{\cA\in\cG(\cS)}E(\cA)/\sim\,,
\]
with equivalence relation given by identifying every point in $E(\cB)$ to its image in $E(\cA)$ under all mappings $B\Sigma_T:E(\cB)\to E(\cA)$ constructed as above, is a CW-complex is equivalent to the following proposition.

\begin{prop}
For a fixed $\cA\in\cG(\cS)$ of rank $n$, the embeddings $B\Sigma_T:E(\cB)\into E(\cA)$ for all cluster morphisms $[T]:\cA\to \cB$ of rank $\ge1$ define a continuous map
\[
	\eta_\cA:\partial E(\cA)=B\simp K(\cA)\to \bigcup_{rk\,\cB<n} \overline e(\cB)
\]
giving the attaching map for the cell $e(\cA)$ in a CW-complex $X(\cS)=\bigcup \overline e(\cA)$.
\end{prop}

\begin{proof} If $n=0$ then $\partial E(\cA)$ is empty and there is nothing to prove. So, suppose that $n>0$ and the proposition holds for numbers $<n$. In particular $\bigcup_{rk\,\cB<n}\overline e(\cB)$ is a CW-complex

The statement is that the maps $B\Sigma_T:E(\cB)\to E(\cA)$ together form a surjective continuous mapping
\[
\bigcup_{[T]:\cA\to \cB}\overline e(\cB)=\coprod E(\cB)/\! \sim\ \onto \partial E(\cA)
\]
and that, furthermore, any two elements which map to the same point in $\partial E(\cA)$ are already identified in the subcomplex $\bigcup_{[T]:\cA\to \cB}\overline e(\cB)\subseteq\bigcup_{rk\,\cB<n}\overline e(\cB)$.

To prove the surjectivity statement, take any point $z\in \partial E(\cA)$. Then $z$ will be in the span of a simplex
\[
Z_\ast:	Z_0\subset Z_1\subset \cdots\subset Z_p
\]
where each $Z_i$ is nonempty. Let $\cB=|Z_0|^\perp$. Then $[Z_0]:\cA\to\cB$ is a cluster morphism of positive rank and the $p$-simplex $Z_\ast$ in $\partial E(\cA)$ is the image of the $p$-simplex
\[
	X_\ast: X_0\subset X_1\subset\cdots\subset X_p
\]
in $E(\cB)$ where $X_0=\emptyset$ and each $X_i$ is the unique partial cluster tilting set in $\cB$ so that 
\[
Z_i=Z_0\cup \sigma_{Z_0}X_i=\Sigma_{Z_0}X_i
\]
and $z=B\Sigma_{Z_0}x$ where $x$ is a point in the simplex spanned by $X_\ast$. Therefore, $\bigcup B\Sigma_T:\bigcup \overline e(\cB)\onto \partial E(\cA)$ is surjective.

Now suppose that $y\in E(\cB')$ maps to the same point $z\in \partial E(\cA)$ under the map induced by $[T]:\cA\to\cB'$. Since each $\Sigma_T$ is an embedding, this implies that $y$ lies in the interior of a simplex of the same dimension as $Z_\ast$, say, $
	Y_\ast:Y_0\subset Y_1\subset\cdots\subset Y_p
$. This implies that
\[
	Z_i=T\cup \sigma_TY_i
\]
In particular, $T\subseteq Z_0$ and $Z_0=T\cup \sigma_TY_0$. In other words, the morphism $[Z_0]:\cA\to\cB$ is the composition of $[T]$ and $[Y_0]:\cB'\to\cB$. By Lemma \ref{lem: Sigma is a functor}, this implies $\Sigma_{Z_0}=\Sigma_T\circ \Sigma_{Y_0}$. Since $\Sigma_T$ is an embedding, the equation
\[
	\Sigma_TY_i=Z_i=\Sigma_{Z_0}X_i=\Sigma_T\Sigma_{Y_0}X_i
\]
implies $Y_i=\Sigma_{Y_0}X_i$ for all $i$. So, $y=B\Sigma_{Y_0}(x)$ and the points $x,y$ are identified in the subcomplex $\bigcup \overline e(\cB)$.
\end{proof}

\subsection{Proof that $X(\cS)=B\cG(\cS)$}\label{ss 4.2: proof that X(S)=BG(S)}
We will show:

\begin{thm}\label{thm: BG(S)=X(S)}
For any finite convex set $\cS$ of real Schur roots, we have a homeomorphism
\[
	X(\cS)\cong B\cG(\cS).
\]
The image of $\overline e(\cA)\subseteq X(\cS)$ in $B\cG(\cS)$, denoted $\overline\vare(\cA)$, is the union of all simplices corresponding to sequences of composable morphisms
\[
	\cA_0\to \cA_1\to\cdots\to \cA_p
\]
where $\cA_0=\cA$. The center of the cell $e(\cA)$ maps to $\cA$ considered as a vertex of $B\cG(\cS)$.
\end{thm}

\begin{rem}\label{rem: orientation of cells is given by signed exceptional sequences}
Since the top dimensional simplices are given by maximal sequences of composable morphisms which in turn are given by signed exceptional sequences for $\cA$, each such sequence will give an orientation for the cell $\overline\vare(\cA)$. 
\end{rem}

The proof of Theorem \ref{thm: BG(S)=X(S)} is based on the following general observation.

\begin{prop}\label{prop: BC is the union of B(X under C)}
The classifying space of any small category $\cD$ is equal to the union of classifying spaces $B(X\backslash \cD)$ of under-categories $X\backslash \cD$ for all $X\in\cD$ modulo the identifications given by all mappings
\[
	Bf^\ast: B(Y\backslash \cD)\to B(X\backslash \cD)
\]
induced by all morphisms $f:X\to Y$ in $\cD$. Furthermore, the image of $B(X\backslash \cD)$ in $B\cD$ is the union of all simplices corresponding to sequences of composable morphisms
\[
	X\to X_1\to X_2\to\cdots\to X_p
\]
and the identity morphism $(X,id_X)\in X\backslash \cD$ maps to the vertex $X$ in $B\cD$.
\qed
\end{prop}

Since this statement follows from the definitions and holds in any category, we leave the proof to the reader.

\begin{lem}\label{lem: A under G(S) is simp+K(A)}
For any object $\cA$ in the category $\cG(\cS)$ we have an isomorphism of categories:
\[
	\cA\backslash \cG(\cS)\cong \simp_+ K(\cA)
\]
given by sending each objects $[T]:\cA\to\cB$ in $\cA\backslash \cG(\cS)$ to the simplex $T$ considered as an object of $\simp_+ K(\cA)$. This inducing a homeomorphism $B(\cA\backslash \cG(\cS))\cong E(\cA)\cong D^n$.
\end{lem}

\begin{proof}
Recall that the objects of $\cA\backslash \cG(\cS)$ are cluster morphisms $[T]:\cA\to\cB$ given by partial (unordered) cluster tilting sets $T=\{T_1,\cdots,T_k\}\subset \cC(\cA)$. There is a unique morphism $[T]\to [S]$ if and only if $T\subseteq S$. Thus $\cA\backslash \cG(\cS)$ is a poset category and the mapping $\cA\backslash \cG(\cS)\to \simp_+K(\cA)$ sending $[T]$ to $T$ gives an isomorphism of partial ordered sets and therefore an isomorphism of poset categories.
\end{proof}

\begin{proof}[Proof of Theorem \ref{thm: BG(S)=X(S)}] Proposition \ref{prop: BC is the union of B(X under C)} and Lemma \ref{lem: A under G(S) is simp+K(A)} imply that
\[
	B\cG(\cS)=\coprod_{\cA\in\cG(\cS)} B(\cA\backslash \cG(\cS))/\!\sim\ \cong \coprod_{\cA\in\cG(\cS)} E(\cA)/\!\sim\ =\bigcup_{\cA\in\cG(\cS)}\overline e(\cA)=X(\cS).
\]
It remains to show that the identifications on the cells $E(\cA)\cong B(\cA\backslash \cG(\cS))$ are the same in $B\cG(\cS)$ and in $X(\cS)$. This is equivalent to showing that the following diagram of functors commutes for any morphism $[T]:\cA\to\cB$.
\[
\xymatrix{%begin xy matrix
\cB\backslash \cG(\cS)\ar[d]_{[T]^\ast}\ar[r] &
	\simp_+K(\cB)\ar[d]^{\Sigma_T}\\
\cA\backslash \cG(\cS) \ar[r]& 
	\simp_+K(\cA)
	}%end xy matrix
\]
But this follows from the fact that the vertical maps are defined by the same formula. Namely, they both take the partial cluster tilting set $X$ in $\cB$ to $T\cup \sigma_TX$ in $\cA$. 

This commuting diagram of categories induces a commuting diagram of classifying spaces showing that $B\cG(\cS)$ and $X(\cS)$ are made from the same pieces pasted together in the same way. So, they are homeomorphic.
\end{proof}

\subsection{Example}\label{ss 4.3: example}
 Suppose that $\cA=\cA(\alpha,\beta)$ where $M_\alpha,M_\beta$ are relative simple projectives. Then the only objects in $\cC(\alpha,\beta)$ are $M_\alpha,M_\alpha[1],M_\beta,M_\beta[1]$ and 
 \[
 \xymatrixrowsep{10pt}\xymatrixcolsep{10pt}
\xymatrix{%begin xy matrix
 && \beta \ar@{-}[dl]\ar@{-}[dr] \\
	K(\alpha,\beta)=& -\alpha\ar@{-}[dr] && \alpha\ar@{-}[dl]\\
	&& -\beta	}%end xy matrix
 \]
 This simplicial complex has 4 edges, 4 vertices and one empty simplex:
\[
\xymatrix{%begin xy matrix
 & \{-\alpha,\beta\} & \beta\ar[l]\ar[r] & \{\alpha,\beta\} \\
	\simp_+K(\alpha,\beta)=& -\alpha\ar[u]\ar[d] &\emptyset\ar[u]\ar[d]\ar[l]\ar[r]\ar[lu]\ar[ru]\ar[ld]\ar[rd]& \alpha\ar[d]\ar[u]\\
	& \{-\alpha,-\beta\} & -\beta\ar[l]\ar[r] & \{\alpha,-\beta\}	}%end xy matrix
 \]
with classifying space $E(\cA(\alpha,\beta))\cong D^2$. This category is isomorphic to the under-category:
\[
% \xymatrixrowsep{15pt}
\xymatrixcolsep{30pt}
\xymatrix{%begin xy matrix
	& \cA(\emptyset) & \cA(\alpha)\ar[l]_{[-\alpha]}\ar[r]^{[\alpha]} &\cA(\emptyset) \\ % top row
	\cA(\alpha,\beta)\backslash \cG(\cS)=
	& \cA(\beta)\ar[u]^{[\beta]}\ar[d]_{[-\beta]} 
	&\cA(\alpha,\beta)\ar[u]_{[\beta]}\ar[d]^{[-\beta]}\ar[l]_{[-\alpha]}\ar[r]^{[\alpha]}\ar[lu]_{[-\alpha,\beta]}\ar[ru]^{[\alpha,\beta]}\ar[ld]_{[-\alpha,-\beta]}\ar[rd]^{[\alpha,-\beta]} % arrows from center
	& \cA(\beta)\ar[d]^{[-\beta]}\ar[u]_{[\beta]}\\
	& \cA(\emptyset) & \cA(\alpha)\ar[l]^{[-\alpha]}\ar[r]_{[\alpha]} & \cA(\emptyset) % bottom row
}%end xy matrix
 \]
 The space $B\cG(\alpha,\beta)\cong X(\alpha,\beta)$ has four cells: $\vare(\cA(\emptyset))$ which is at each of the four vertices in the diagram, $\vare(\cA(\alpha))$ which is the interior of the top and bottom rows, $\vare(\cA(\beta))$ which is the interior of the left and right columns and $\vare(\cA(\alpha,\beta))$ which is the interior of the square. So,
 \[
 	X(\alpha,\beta)\cong B\cG(\alpha,\beta)=S^1\times S^1
 \] 
 is a torus.

\subsection{Semi-invariant labels}\label{ss 4.4: semi-invariants}

One of the key properties of the picture space $X(\cS)$ is that it has a ``normally oriented'' codimension one subcomplex
\[
	D(\cS)=\bigcup_{\beta\in\cS} D(\beta) 
\]
where each $D(\beta)$ is locally the support of the virtual semi-invariant with det-weight $\beta$ (Definition \ref{def: det semi-inv and supports}). Using the categorical version of $X(\cS)$, these subspaces are easy to describe.

\begin{defn}
For any $\beta\in\cS$, let $D(\beta)\subseteq B\cG(\cS)$ be the union of all simplices given by composable sequences of morphisms
\[
	\cA_0\to\cA_1\to\cdots\to \cA_p
\]
where $M_\beta\in \cA_p$. Then $D(\cS)=\bigcup D(\beta)$ is the union of all simplices given by sequences of morphisms as above where $\cA_p$ is nonzero.
\end{defn}

It follows directly from this definition that the complement of $D(\cS)$ in $B\cG(\cS)\cong X(\cS)$ is the open star of $\cA(\emptyset)$ which is, by definition, the set of all points so that the barycentric coordinate of $\cA(\emptyset)$ is positive. This is a contractible space with deformation retraction to the vertex $\cA(\emptyset)$ given by linear deformation of barycentric coordinates.

In the universal covering $\tilde X(\cS)$ of $X(\cS)$, the complement of the inverse image $\tilde D(\cS)$ of $D(\cS)$ in $\tilde X(\cS)$ is a disjoint union of contractible spaces, one for each element of the fundamental group $G(\cS)$ of $X(\cS)$. This gives is a locally constant function
\[
	g: \tilde X(\cS)\backslash \tilde D(\cS)\to G(\cS)
\]
which has the following property.

For any root $\beta\in\cS$, let $x_t$ be the path in $X(\cS)$ given by the cell $E(\cA(\beta))\cong B\cA(\beta)\backslash\cG(\cS)$ going from left to right along the path 
\[
	\cA(\emptyset)\xleftarrow{[-\beta]} \cA(\beta)\xrarrow{[\beta]}\cA(\emptyset).
\]
This path intersects $D(\beta)$ only in its midpoint $\cA(\beta)$. Since this represents the generator $x(\beta)$ of $\cG(\cS)$, any lifting $\tilde{x_t}$ of this path to $\tilde X(\cS)\cong B\cG(\cS)$ will have the property that
\[
	g(\tilde x_1)=g(\tilde x_0)x(\beta).
\]

We will now determine the relationship between the subcomplex $D(\beta)\subseteq X(\cS)$ and the subspace $D_{\alpha_\ast}(\beta)\subseteq \RR\alpha_\ast$ defined in Theorem \ref{Stability theorem for virtual semi-invariants}. Suppose that $\cA=\cA(\alpha_1,\cdots,\alpha_n)$ where $\alpha_i\in\cS$. Then, for each positive root $\beta\in\Phi_+(\cA)$, we recall that 
\[
	D_{\alpha_\ast}(\beta) = \{v\in\RR\alpha_\ast\cong \RR^n\,|\, \brk{v,\beta}=0 \text{ and } \brk{v,\beta'}\le 0\ \forall \beta'\subseteq\beta,\beta'\in \Phi_+(\cA)\}
\]
This is a closed convex subset of the hyperplane $\{v\in\RR\alpha_\ast\,|\,\brk{v,\beta}=0\}$. This hyperplane has a normal orientation. The \emph{positive side} is the set
\[
	\{v\in\RR\alpha_\ast\,|\,\brk{v,\beta}>0\}.
\]
For example, $\beta$ is on the positive side of $D_{\alpha_\ast}(\beta)$. One point in $\RR\alpha_\ast$ which is on the positive side of all of these hyperplanes is the dimension vector of the sum of all projective objects.

We recall that $D_{\alpha_\ast}(\beta)$ contains the dimension vector of any object $M\in\cC(\cA)$ with the property that 
\[
\Hom_\Lambda(|M|,M_\beta)=0=\Ext^1_\Lambda(|M|,M_\beta),
\]
i.e., $|M|\in \,^\perp M_\beta$. This implies that, given any cluster tilting set $T=(T_1,\cdots,T_n)$ in $\cC(\cA)$, with corresponding $c$-vectors $(-\gamma_1,\cdots,-\gamma_n)$, we have
\[
	\undim T_i\in D_{\alpha_\ast}(|\gamma_j|)
\]
for all $i\neq j$. Furthermore, $\undim T_j$ is on the positive or negative side of $D_{\alpha_\ast}(|\gamma_j|)$ depending on whether $\gamma_j$ is positive or negative, respectively.

For any $n$-dimensional simplicial complex we define a \emph{normal orientation} on an $n-1$ simplex $\tau$ to be the assignment of a sign ($+$ or $-$) to each $n$-simplex containing $\tau$ as a face. A \emph{normal orientation} of an $n-1$ dimensional subcomplex of an $n$-dimensional simplicial complex is defined to be a normal orientation of each of its $n-1$ simplices. We do not assume any consistency between orientations of adjacent $n-1$ simplices.

\begin{defn}
Let $\cA\in\cG(\cS)$ of rank $n$ and let $\beta\in \Phi_+(\cA)$. We define $L_\cA(\beta)\subset K(\cA)$ to be the normally oriented codimension one subcomplex  consisting of simplices all of whose vertices lie in $\cA\cap \,^\perp M_\beta$. The normal orientation is given in the discussion above. Namely, an $n-1$-simplex $T=\{T_1,\cdots,T_n\}$ with one face $\partial_i T$ in $\,^\perp M_\beta$ has positive sign if the corresponding vector $\gamma_i$ is positive (the $c$-vector $-\gamma_i$ is negative). Take the full subcategory $\simp_+L_\cA(\beta) \subset \simp_+K(\cA)$ which is normally oriented when considered as a simplicial complex. Denote its classifying space by
\[
	D_\cA(\beta)=B\simp_+L_\cA(\beta)\subset B \simp_+K(\cA)=E(\cA).
\]
\end{defn}

For a fixed $\cA$ with rank $n$, the space
\[
	L(\cA)=\bigcup_{\beta\in \Phi_+(\cA)} B\simp L_\cA(\beta)\subset B\simp K(\cA)\cong S^{n-1}
\]
is the picture for $\cA$ as defined in \cite{IOTW4} and $B\simp L_\cA(\beta)$ is the normally oriented subset labeled $\beta$. The following proposition shows that the spaces $D(\beta)\subseteq B\cG(\cS)$ play the analogous role. So, the union $D(\cS)=\bigcup D(\beta)$ is a generalization of a picture.

\begin{prop}\label{prop: DA(b)=e(A) cap D(b)}
$D_\cA(\beta)\subseteq E(\cA)$ is the inverse image of $\overline \vare(\cA)\cap D(\beta)$ under the epimorphism $E(\cA)\onto \overline \vare(\cA)$. Furthermore, the normal orientation of $D_\cA(\beta)$ is such that each embedding of the 1-cell $E(\cA(\beta))$ in $E(\cA)$, oriented by the path $[-\beta]^{-1}[\beta]$ passes from the negative to the positive side of $D_\cA(\beta)$.
\end{prop}

\begin{proof}
By Theorem \ref{thm: BG(S)=X(S)} and the definition of $D(\beta)$, the intersection $\overline \vare(\cA(\alpha_\ast))\cap D(\beta)$ is the union of all simplices in $B\cG(\cS)$ corresponding to sequences of morphisms $\cA_0\to\cdots\to \cA_p$ starting at $\cA_0=\cA(\alpha_\ast)$ and ending at $\cA_p=\cA(\beta)$. If the rank of $\cA(\alpha_\ast)$ is $n$, the composition of these morphisms is a morphism
\[
	[T]:\cA(\alpha_\ast)\to \cA(\beta)
\]
where $T=(T_1,\cdots,T_{n-1})$ is a cluster tilting set in $\cA(\alpha_\ast)\cap M_\beta^\perp$. Such a cluster tilting set corresponds to a maximal simplex in $L_{\cA(\alpha_\ast)}(\beta)$ and the sequence of objects $\cA_i$ correspond to faces of this simplex starting with the empty face. So, the inverse image of $\overline \vare(\cA)\cap D(\beta)$ in $E(\cA)$ is $D_\cA(\beta)=B\simp_+L_{\cA}(\beta)\subset B\simp_+K(\cA)$.

Conversely, any simplex in $D_\cA(\beta)$ is a chain of faces of a maximal simplex in $L_\cA(\beta)$. Such a chain corresponds to an ordered cluster tilting set in $\cA\cap M_\beta^\perp$ which corresponds to a maximal chain of morphisms $\cA\to\cdots\to\cA(\beta)$ which is a maximal simplex in $\overline\vare(\cA)\cap D(\beta)$.

The completion of the cluster tilting set $T$ is given by the composition of $[T]:\cA\to\cA(\beta)$ with $[M_\beta]:\cA(\beta)\to\cA(\emptyset)$ which is
$
	[M_\beta]\circ[T]=[T,\sigma_TM_\beta]:\cA\to \cA(\emptyset)
$. The cluster tilting set $(T_1,\cdots,T_{n-1},\sigma_TM_\beta)$ has last $c$-vector $-\gamma_n=-\beta$ since $\brk{\undim T_i,\beta}=0$ and
\[
	\brk{\undim\sigma_TM_\beta,\beta}=\brk{\beta,\beta}>0
\]
Therefore, the normal orientation of $D_\cA(\beta)$ assigns a positive sign to the maximal simplex $\{T_1,\cdots,T_n,M_\beta\}$. But this is equivalent to saying that $[-\beta]^{-1}[\beta]$ goes through $D_\cA(\beta)$ in the positive direction as claimed.
\end{proof}

\subsection{Cellular chain complex for $X(\cS)$}\label{ss 4.5: cellular chain complex}

We recall that the cellular chain complex of any $CW$-complex $X$ is:
\[
	\cdots \to C_n(X)\xrarrow{d_n} C_{n-1}(X) \to \cdots \to C_1(X)\xrarrow{d_1} C_0(X)\to 0
\]
where $C_n(X)$ is the free abelian group generated by the $n$-cells of $X$ with some chosen orientation for each cell. The boundary map $d_n:C_n(X)\to C_{n-1}(X)$ is given by an integer matrix whose $ij$ coordinate is the incidence number of the composition
\[
	S^{n-1}\xrarrow{\eta_j} X^{n-1}\xrarrow{\pi_i} S^{n-1}
\]
where $X^{n-1}$ is the $n-1$ skeleton of $X$, $\eta_j$ is the attaching map of the $j$th $n$-cell of $X$ and $\pi_i$ is the map which collapses all cells in $X^{n-1}$ to a point except for the $i$th $n-1$-cell.

In the case $X=X(\cS)$, where $\cS$ is a finite convex set of real Schur roots, the generators of $C_n(X)$ are oriented wide categories $\cA=\cA(\alpha_1,\cdots,\alpha_n)$ where $\alpha_i\in\cS$. We denote this element $[\cA]\in C_n(X)$. The orientation is given by the ordering of the hom-orthogonal roots $\alpha_i$ which span $\cA(\alpha_\ast)$. Any odd permutation of the $\alpha_i$ will change the sign of the generator. For example $[\cA(\alpha_2,\alpha_1)]=-[\cA(\alpha_1,\alpha_2)]$.

\begin{thm}\label{thm: chain complex for X(S)}
The boundary map $d_n:C_n(\cS)\to C_{n-1}(\cS)$ is given on each oriented generator $\cA=\cA(\alpha_1,\cdots,\alpha_n)$ by
\[
	d_n[\cA]=\sum_{\beta\in \Phi_+(\cA)\text{ not projective}} \det(c_{ij}) [\cA\cap M_\beta^\perp]
\]
where the sum is over all nonprojective exceptional roots $\beta\in\Phi_+(\cA)\subseteq\cS$. The sign $\det(c_{ij})=\pm1$ is the determinant of the unique integer matrix $(c_{ij})$ satisfying
\[
	\beta_i=\sum_{j=1}^n c_{ij}\alpha_j
\]
for all $1\le i\le n$ where $\cA\cap \beta^\perp=\cA(\beta_1,\cdots,\beta_{n-1})$ is any chosen orientation of $\cB=\cA\cap M_\beta^\perp$ and $\beta_n=\beta$.\end{thm}

\begin{proof}
By Theorem \ref{thm: BG(S)=X(S)}, the $n$-cell $\varepsilon(\cA)$ in $X(\cS)$ is the union of $n$-simplices $\cA_0\to \cA_1\to\cdots\to \cA_n$ where $\cA_0=\cA$. The first morphism $\cA\to \cB=\cA_1$ is given by a single exceptional object $[M_\beta]$ in the cluster category of $\cA$ and the $n-1$ simplex $\cB=\cA_1\to\cdots\to \cA_n$ is part of the $n-1$ cell $\varepsilon(\cB)$. Every simplex is oriented by the ordering of its vertices. Since each maximal chain of composable morphisms $\cA=\cA_0\to\cdots\to \cA_n$ is given by a signed exceptional sequence, each such sequence gives an orientation of $\varepsilon(\cA)$. \vs2

Claim: The corresponding sequence of dimension vectors $(\alpha_1,\cdots,\alpha_n)$ is unique up to invertible integer matrix tranformation. I.e., $\beta_i=\sum c_{ij}\alpha_j$ where $(c_{ij})\in GL(n,\ZZ)$ for any other such sequence $(\beta_i)$.

Proof of Claim: Any two exceptional sequences can be transformed into each other by braid moves. Each braid move changes the sequence of dimension vectors by transposing two and adding a multiple of one to the other. The signs in a signed exceptional sequence can be changed by multiplication by a diagonal matrix with entries $\pm1$. In all cases, the dimension vectors change by an integer matrix of determinant $\pm1$.
\vs2

Suppose we have a fixed orientation of the $n$-cell $\varepsilon(\cA)$. Then which morphisms $[M_\beta]:\cA\to \cB$ do we have? By definition of cluster morphism, there is one such morphism for every (isomorphism class of) indecomposable object $M_\beta$ of the cluster category of $\cA$. These objects have target $\cB=\cA\cap M_\beta^\perp$. Each wide subcategory $\cB\subseteq\cA$ of rank $n-1$ occurs in this way and $M_\beta$ is uniquely determined by $\cB$ except in the case when $M_\beta$ is projective in which case $[M_\beta[1]]=[M_{-\beta}]$ is also a morphism $\cA\to\cB$.

When $M_\beta$ is not projective, the incidence number of $\varepsilon(\cA)$ with $\varepsilon(\cB)$ is $\pm1$ and the sign is determined by the choice of orientation of both $\cA$ and $\cB$. The orientation of $\cB$ is specified by an $n-1$ simplex: $\cB=\cB_1\to \cB_2\to\cdots\to \cB_n=0$ which is given by a signed exceptional sequence $(\beta_1,\cdots,\beta_{n-1})$. Appending the morphism $[M_\beta]:\cA\to\cB$ gives the signed exceptional sequence $(\beta_1,\cdots,\beta_{n-1},\beta)$. If $(c_{ij})$ is the comparison matrix of this sequence with $(\alpha_1,\cdots,\alpha_n)$ then $\det(c_{ij})$ is the incidence number of $[\cA]$ with $[\cB]$. 

For each projective object $P=M_\beta\in\cA$ there are two objects in the cluster category: $P$ and $P[1]=M_{-\beta}$. This gives two morphisms $[M_{\pm \beta}]:\cA\to \cB$. For any fixed orientations $(\alpha_\ast),(\beta_\ast)$ of $\cA,\cB=\cA\cap P^\perp$ these two morphisms have opposite sign since the sign of the last vector $\beta_n=\pm \beta$ changes. Therefore, the incidence number of $[\cA]$ and $[\cA\cap P^\perp]$ is zero. This proves the formula for $d_n:C_n(X(\cS))\to C_{n-1}(X(\cS))$ for any finite convex set $\cS$.
\end{proof}

\section{Acknowledgements} The authors wish to thank Kent Orr and Jerzy Weyman for the many years that we spent discussing semi-invariant pictures and their possible meaning. The first author acknowledges support of National Security Agency Grant \#H98230-13-1-0247 and the second author acknowledges support by National Science Foundation Grants \#DMS-1103813 and \#DMS-0901185 during the work reported in this paper.

%%%%%%%%%%%%%%%%%%%%%%%%%%%%%%%%%%%
\end{document}